\documentclass[10pt,reqno]{amsart}
\usepackage{amsmath, amsthm, amssymb, enumerate}
\usepackage[usenames]{color}
\usepackage{eepic,epic}
\usepackage{esint}

\newtheorem{thm}{Theorem}[section]
\newtheorem{cor}[thm]{Corollary}
\newtheorem{lem}[thm]{Lemma}
\newtheorem{prop}[thm]{Proposition}
\theoremstyle{definition}

\theoremstyle{remark}
\newtheorem{rem}[thm]{Remark}
\numberwithin{equation}{section}

\newcommand{\R}{\mathbb{R}}

\begin{document}

\title[]
{}


\title[Semilinear elliptic equations with the pseudo-relativistic operator]
{Semilinear elliptic equations with the pseudo-relativistic operator on a bounded domain}

\author{Woocheol Choi}
\address{W. Choi, Department of Mathematics Education, Incheon National University, Incheon 22012, Korea}
\email{choiwc@inu.ac.kr}

\author{Younghun Hong}
\address{Y. Hong, Department of mathematics, Yonsei University, Seoul 03722, Korea}
\email{younghun.hong@yonsei.ac.kr}

\author{Jinmyoung Seok}
\address{J. Seok, Department of Mathematics, Kyonggi University, Suwon 16227, Korea}
\email{jmseok@kgu.ac.kr}

\begin{abstract}
We study the Dirichlet problem for the semilinear equations involving the pseudo-relativistic operator on a bounded domain,
\begin{equation*}
(\sqrt{-\Delta + m^2} - m)u =|u|^{p-1}u \quad \textrm{in}~\Omega,\end{equation*}
with the Dirichlet boundary condition $u=0$ on $\partial \Omega$. Here, $p \in (1,\infty)$ and the operator $(\sqrt{-\Delta + m^2} - m)$ is defined in terms of spectral decomposition. 
In this paper, we investigate existence and nonexistence of a nontrivial solution, depending on the choice of $p$, $m$ and $\Omega$.
Precisely, we show that $(i)$ if $p$ is not $H^1$ subcritical ($p \geq \frac{n+2}{n-2}$) and $\Omega$ is star-shaped, the equation has no nontrivial solution for all $m > 0$; $(ii)$ if $p$ is not $H^{1/2}$ supercritical ($1 <p \leq \frac{n+1}{n-1}$), then there exists a least energy solution for all $m>0$ and any bounded domain $\Omega$; $(iii)$ finally, in the intermediate range  ($\frac{n+1}{n-1}<p<\frac{n+2}{n-2}$), the problem has a nontrivial solution, provided that $m$ is sufficiently large and the problem 
\begin{equation}
-\Delta u = |u|^{p-1}u \quad \textrm{in}~\Omega, \qquad u =0\quad \textrm{on}~\partial \Omega
\end{equation}
admits a non-degenerate nontrivial solution, for example, when $\Omega$ is a ball or an annulus.
\end{abstract}

\maketitle

 \\
\section{Introduction}
Let $n \geq 3$ and $p > 1$, and let $\Omega \subset \R^n$ be a smooth bounded domain. This paper is devoted to the study of the Dirichlet problem:
\begin{equation}\label{eq-main}
\left\{\begin{aligned} 
\mathcal{P}_mu &=|u|^{p-1}u &&\textrm{in}~\Omega,\\
u&=0&&\textrm{on}~\partial \Omega,
\end{aligned}
\right.\end{equation}
where $\mathcal{P}_m$ is the pseudo-relativistic operator given by $\sqrt{-\Delta + m^2} - m$ with a particle mass $m > 0$. 
On a bounded domain, there are several ways to define a nonlocal square root operator with the zero Dirichlet condition unlike the whole domain $\R^n$ case. 
In this paper, we adopt the definition using the spectral decomposition (see Section 2 below for the precise definition).

The pseudo-relativistic operator $\mathcal{P}_m$ appears in the relativistic theory of the quantum mechanics to describe stellar objects, e.g., boson and fermion stars, in astrophysics. We refer the readers to the fundamental works in \cite{LY1, LY2} on the stability of relativistic matter. Based on this physical motivation, there have been many studies conducted on nonlinear problems involving the pseudo-relativistic operator. We refer the readers to \cite{ACV0, ACV, CiS, CN1, CN2, M} for semilinear elliptic equations with the Hartree nonlinearity; \cite{A, ChS, TWY} for those with power-type nonlinearity; \cite{MNO, MN, Ts} for time evolution equations. The operator $\mathcal{P}_m$ on a bounded domain appears in the work of Abou Salem, Chen and Vougalter\,\cite{ACV0, ACV}, in which the authors investigated the semi-relativistic Schr\"odinger-Poisson system on a bounded domain.

The operator $\mathcal{P}_m$ is called {\it pseudo-relativistic}, because it interpolates the non-relativistic and the ultra-relativistic schr\"odinger operators,
i.e., $2m\mathcal{P}_m$ formally converges to $-\Delta$ as $m\to\infty$ as well as $\mathcal{P}_m$ converges to $\sqrt{-\Delta}$ as $m\to 0^+$. 
When $m$ is near infinity, the nonlinear problem \eqref{eq-main} thus can be thought of as a perturbation of the Lane-Emden equation, 
\begin{equation}\label{eq-2}
\left\{\begin{aligned} -\Delta u &= |u|^{p-1}u && \textrm{in}~\Omega,\\
u&=0&&\textrm{on}~\partial \Omega,
\end{aligned}
\right.
\end{equation}
which arises in astrophysics for describing the pressure and the density of a gas sphere in a hydrostatic equilibrium. Here, the power of nonlinearity $p$ stands for the polytropic index (see \cite{Cha, Ki}). On the other hand, as $m$ goes to $0^{+}$, the equation \eqref{eq-main} converges to the nonlocal problem 
\begin{equation}\label{eq-1}
\left\{\begin{aligned} \sqrt{-\Delta}\,u &= |u|^{p-1}u&&\textrm{in}~\Omega,
\\
u&=0&&\textrm{on}~\partial \Omega,
\end{aligned}
\right.
\end{equation}
where $\sqrt{-\Delta}$ corresponds to the square root of the Laplacian $-\Delta$ defined in \cite{CT}.

The purpose of this paper is to exploit existence and non-existence of a nontrivial solution to the pseudo-relativistic equation \eqref{eq-main}
and compare them with the results on the two limit equations \eqref{eq-2} and \eqref{eq-1} as $m\to\infty$ and $m\to 0^+$ respectively. 
Before stating our main theorems, we briefly recall the existence results for the limit problems.  
An important remark is that the above two problems have the different critical exponents, and thus the ranges of the exponent $p$ for existence and nonexistence of a nontrivial solution are different. 
Precisely, it is known that the problem \eqref{eq-2} has a positive solution if $1 < p < \frac{n+2}{n-2}$, but it has no nontrivial solution if $p \geq \frac{n+2}{n-2}$ and $\Omega$ is starshaped (see, e.g., \cite{QS}). 
The same results hold for \eqref{eq-1} if we replace $\frac{n+2}{n-2}$ with $\frac{n+1}{n-1}$ (see \cite{CT, T}).  As is well-known, the two critical exponents $\frac{n+2}{n-2}$ and $\frac{n+1}{n-1}$ are related to the critical Sobolev embeddings $H_0^1 (\Omega) \hookrightarrow L^{2n/(n-2)}(\Omega)$ and $H_0^{1/2} (\Omega) \hookrightarrow L^{2n/(n-1)}(\Omega)$.

Coming back to the problem \eqref{eq-main}, we now introduce the notion of a solution for the equation. In Section \ref{prelim}, it will be shown that the operator $\mathcal{P}_m$ is bounded from $H^1_0(\Omega)$ to $L^2(\Omega)$.
Thus, it is natural to say that a real valued function $u$ on $\Omega$ is a (strong) solution of $\eqref{eq-main}$ if $u \in H^1_0(\Omega)$, $|u|^{p-1}u \in L^2(\Omega)$ and $\mathcal{P}_mu = |u|^{p-1}u$ almost everywhere. 

Our first main theorem establishes existence of a positive nontrivial solution to \eqref{eq-main}
for all $m > 0$ and $1 < p \leq \frac{n+1}{n-1}$. This shall be achieved by searching for a nontrivial critical point of the action integral
\[
I_m(u) := \int_{\Omega}\frac12|(-\Delta +m^2)^{1/4}u|^2-\frac{m}{2}u^2+\frac{1}{p+1}|u|^{p+1}\,dx,
\]
whose Euler-Lagrange equation is \eqref{eq-main} (see Section \ref{prelim}).
\begin{thm}\label{thm-main-1} 
Suppose that $1<p\leq \frac{n+1}{n-1}$ and fix any $m > 0$. Then, the problem \eqref{eq-main} admits a positive nontrivial solution.
Moreover, it is contained in $C^{2,\alpha}(\overline{\Omega})$ for some $0<\alpha <1$ and minimizes the action integral $I_m$ among all the nontrivial solutions of \eqref{eq-main}.
\end{thm}

The main novelty of Theorem \ref{thm-main-1} lies on the $H^{1/2}$ critical case ($p = \frac{n+1}{n-1}$), which is twofold. First of all, a non-trivial solution is constructed without any assumption on the size of $m>0$. Indeed, from existence and non-existence of a non-trivial solution for the limit equations \eqref{eq-2} and \eqref{eq-1}, it is natural to speculate that when $p = \frac{n+1}{n-1}$, a non-trivial solution exists if $m$ is large enough, while a nontrivial solution does not exists if $m$ is sufficiently small. Such a speculation turns out to be true for a similar problem
\begin{equation}\label{eq-whole-domain}
\mathcal{P}_m u + \mu u = |u|^{p-1}u\quad \textrm{in}~\mathbb{R}^n
\end{equation}
on the whole domain. Precisely, in \cite{CHS}, it is shown that the equation \eqref{eq-whole-domain} with $p = \frac{n+1}{n-1}$ admits a nontrivial solution when $m$ is sufficiently large, but it does not admit any nontrivial solution in $L^\infty \cap H^{1/2}$ when $m \leq \mu$. However, Theorem \ref{thm-main-1} shows that unlike the whole domain case, the problem on a bounded domain may admit a nontrivial solution even for small $m$. Another new feature is that a solution is constructed as a minimizer for the action functional in the $H^{1/2}$ critical case ($p = \frac{n+1}{n-1}$), while for the problem \eqref{eq-whole-domain} on the whole domain, a solution is constructed via a contraction mapping argument with no further information about its variational character. Indeed, being a minimizer is in general crucial to explore interesting properties of a solution like uniqueness, symmetries and asymptotical behaviors. In the forthcoming work, we will investigate such properties of a solution in Theorem \ref{thm-main-1}.

The main strategy for proving Theorem \ref{thm-main-1} is to find an energy level $c$ at which a (PS) sequence of $I_n$, i.e., a sequence $\{u_n\}$ such that $I_m'(u_n) \to 0$ and $I_m(u_n) \to c$ as $n \to \infty$, is precompact. 
We shall observe that the energy gain raised by the term $\sqrt{-\Delta +m^2}\,u$ is overwhelmed by the energy loss resulting from the term $-mu$ for any $m > 0$, 
which consequently gets rid of the possibility of concentrating bubbles and allows us to obtain a nontrivial solution of \eqref{eq-main}.


Next, we state our second main result, which concerns about nonexistence. 
\begin{thm}\label{thm-ne} Assume that $p \geq \frac{n+2}{n-2}$, and $\Omega$ is star-shaped. 
Then, there exists no nontrivial $L^\infty$ bounded solution to \eqref{eq-main} for any $m > 0$.
\end{thm}
In order to obtain this result, we shall devise an anisotropic Pohozaev identity for the problem \eqref{eq-lm}.
We remark that Theorem \ref{thm-ne} does not cover the range $p \in (\frac{n+1}{n-1},\, \frac{n+2}{n-2})$.
In fact, at the limit $m \to \infty$ the critical exponent of the limit problem \eqref{eq-2} is not $\frac{n+1}{n-1}$, but rather $\frac{n+2}{n-2}$. 
Thus, one would expect existence of solutions even for $p \in (\frac{n+1}{n-1},\, \frac{n+2}{n-2})$, at least when $m > 0$ is sufficiently large.

The last main result of this paper confirms this if we further assume the existence of a non-degenerate nontrivial solution to the limit equation \eqref{eq-2}.
We say a solution $u_{\infty} \in H_0^1 (\Omega)$ to \eqref{eq-2} is \textit{non-degenerate} if the linearized equation of \eqref{eq-2} at $u_\infty$, i.e.,
\begin{equation}
\left\{ \begin{array}{rll}
-\Delta v =&p |u_{\infty}|^{p-1} v &\quad \textrm{in}~\Omega,
\\
v=&0&\quad \textrm{on}~\partial \Omega
\end{array}
\right.
\end{equation}
admits only the trivial solution $v=0$ in $H_0^1 (\Omega)$.
\begin{thm}\label{thm-3} 
Let $\frac{n+1}{n-1} < p < \frac{n+2}{n-2}$. Suppose that the equation \eqref{eq-2} admits a nontrivial non-degenerate solution $u_\infty$. 
Then, there exists a large $m_0 >0$ such that if $m > m_0$, the problem \eqref{eq-main} has a nontrivial solution $u_m \in W^{1,n}_0(\Omega)$ which satisfies 
\begin{equation}
\|(2m)^{\frac{1}{p-1}}u_m - u_{\infty}\|_{W_0^{1,n} (\Omega)} \leq \left\{\begin{aligned} &\frac{C}{m^2} \quad \text{if } p > 2 \\ & \frac{C}{m} \quad \text{if } 1 < p \leq 2,\end{aligned}\right.
\end{equation}
where $C$ is a positive constant independent of $m$.
\end{thm}
\begin{rem}
The existence of a nontrivial non-degenerate solution to \eqref{eq-2} is known, for example, if one of the following holds:
\begin{enumerate}[$(i)$]
\item The dimension is two and the domain is convex and $p \in (1, \infty)$ \cite{Lin}.
\item the domain $\Omega$ is a ball or an annulus; and $p \in (1,\, \frac{n+2}{n-2})$ \cite{AP, BCGP}.
\item the domain $\Omega$ is convex and is symmetric about coordinate axes after a rigid motion; and $p$ is smaller but sufficiently close to $\frac{n+2}{n-2}$ \cite{G}.
\end{enumerate}
\end{rem}
We remark that the functional $I_m$ is also well defined on $H^1_0(\Omega)$ for $p \in (\frac{n+1}{n-1},\, \frac{n+2}{n-2})$, but the variational approach is not suitable to construct a nontrivial solution,
because it does not seem possible to obtain the $H^1$ boundedness of a (PS) sequence, due to the fact that $H^{1/2}$ norm appearing in $I_m$ cannot control $H^1$ norm.

We construct a non-trivial solution for \eqref{eq-main} by the alternative approach proposed in our earlier work \cite{CHS}. The key ingredient is existence of a non-degenerate solution to the limit equation \eqref{eq-2}, which is assumed in the theorem, because it allows us to make use of the contraction mapping principle to find a solution to \eqref{eq-main} near the non-degenerate solution. 
In this analysis, the H\"ormander-Mikhlin Theorem on Fourier multiplier operators is helpful to cover the full range $1<p<\frac{n+2}{n-2}$, because it gives some $L^p$ estimates for the pseudo-relativistic operator. 
We remark that as mentioned in \cite{CHS}, only a shorter range $1<p\leq \frac{n}{n-2}$ can be included without the H\"ormander-Mikhlin Theorem. In our setting, the domain under consideration is bounded so the standard H\"ormander-Mikhlin theorem cannot be applied directly. 
We resolve this difficulty by invoking the generalized H\"ormander-Mikhlin theorem in Duong, Sikora and Yan \cite{DSY}, which includes that on a bounded domain.

The rest of this paper is organized as follows. In Section \ref{prelim}, we present some preliminary results concerning relevant function spaces,
a precise definition of the operator $\mathcal{P}_m$ as well as a localization of  nonlocal operator $\mathcal{P}_m$.
Well-defined notions of solution to \eqref{eq-main} are also given.  
In Section \ref{sec-nonex}, we prove the nonexistence result of Theorem \ref{thm-ne} by establishing an anisotropic Pohozaev identity.  
In Section \ref{sec-ex-sub} and Section \ref{sec-ex-crit}, we prove the existence result of Theorem \ref{thm-main-1}.
Finally, Section \ref{sec-ex-sup} is devoted to dealing with the $H^{1/2}$ supercritical case with $m$ near infinity. 
\smallskip

\noindent \textbf{Notations.}
\noindent Here, we list some notations that will be used throughout the paper.

\noindent - $B_n (x,r) = \{ y \in \mathbb{R}^n : |y-x| < r\}.$

\noindent - For a domain $D \subset \mathbb{R}^n$, the map $\nu = (\nu_1, \cdots, \nu_n): \partial D \to \mathbb{R}^n$ denotes the outward unit normal vector on $\partial D$.

\noindent - $|S^{n-1}| = 2\pi^{n/2}/\Gamma(n/2)$ denotes the Lebesgue measure of $(n-1)$-dimensional unit sphere $S^{n-1}$.

\noindent - The letter $z$ represents a variable in the $\mathbb{R}^{n+1}$. Alternatively, this is written as $z = (x,t)$ with $x \in \mathbb{R}^n$ and $t \in \mathbb{R}$.

\noindent - $\mathbb{R}^{n+1}_{+} = \{ (x_1, \cdots, x_{n+1}) \in \mathbb{R}^{n+1} : x_{n+1} >0\}.$

\noindent - $dS$ stands for the surface measure. In addition, a subscript attached to $dS$ (such as $dS_x$ or $dS_z$) denotes the variable of the surface.

\noindent - $C > 0$ is a generic constant that may vary from line to line.

\section{Preliminary results}\label{prelim}
In this section, we prepare some preliminary results that are relevant to our problem. First, we provide the definition of the pseudo-relativistic operator, and describe the notions of solutions. 
Next, we recall the extension problem, which localizes the nonlocal equation \eqref{eq-main}, and we provide regularity results for \eqref{eq-main}. 
In the last of the section, we recall the sharp Sobolev-trace inequality and some related entire problems, and we prove an interesting inequality in Proposition \ref{lem-ce-1}, which is essential for proving the existence result of Theorem \ref{thm-main-1}. 

\subsection{pseudo-relativistic operator and notions of solutions}

Here, we define the square root operator $\sqrt{-\Delta +m^2}$ with the zero Dirichlet condition for $m \geq 0$. 
Let $\{\lambda_n,\, \phi_n\}_{n=1}^\infty$ be the complete $L^2$ orthonormal system of eigenvalues and eigenfunctions of
\[
\left\{\begin{aligned}
-\Delta \phi &= \lambda \phi &&\text{in } \Omega \\
\phi &= 0 &&\text{on } \partial\Omega.
\end{aligned}\right.
\]
For any $u \in L^2(\Omega)$, we have the unique decomposition $u = \sum_{i=1}^\infty c_i\phi_i$ in $L^2(\Omega)$. 
It is well known that 
\[
H^1_0(\Omega) = \left\{ u \in L^2(\Omega) ~\Big|~ \sum_{i=1}^\infty c^2_i\lambda_i < \infty \right\}. 
\]
For $m \geq 0$, we define 
\[
\sqrt{-\Delta+m^2}\,u := \sum_{i=1}^\infty c_i\sqrt{\lambda_i+m^2}\,\phi_i.
\]
and denote $\mathcal{P}_m :=\left(\sqrt{-\Delta +m^2} -m\right)$.  Then, the operator $\mathcal{P}_m$ is a continuous map from $H^1_0(\Omega)$ to $L^2(\Omega)$.
Let $H^{1/2}_0(\Omega)$ be the set $\{u\in L^2(\Omega) ~|~ \sum_{i=1}^\infty c_i^2\lambda_i^{1/2} < \infty,\, u = \sum_{i=1}^\infty c_i\phi_i \}$.
For any $p > 1$, we say a function $u$ is a \textit{(strong) solution} of \eqref{eq-main} if $u$ and $|u|^{p-1}u$ belong to $H^1_0(\Omega)$ and $L^2(\Omega)$ respectively, and they satisfy \eqref{eq-main} almost everwhere in $\Omega$. 
Similarly, one can define a one-fourth power operator of the pseudo-relativistic operator by
\[
(-\Delta+m^2)^{1/4}u := \sum_{i=1}^\infty c_i(\lambda_i+m^2)^{1/4}\phi_i
\]
We say that a function $u$ is a \textit{weak solution} of \eqref{eq-main} if $u$ and $|u|^{p-1}u$ belong to $H^{1/2}_0(\Omega)$ and $L^{\frac{2n}{n+1}}(\Omega)$ respectively, and they satisfy
\[
\int_{\Omega}(-\Delta+m^2)^{1/4}u\cdot (-\Delta+m^2)^{1/4}v-muv +|u|^{p-1}uv\,dx = 0, \quad \text{for all } v \in H^{1/2}_0(\Omega).
\]

It is obvious that any strong solution is a weak solution and any weak solution $u$ with the property $u \in H^1_0(\Omega),\, |u|^{p-1}u \in L^2(\Omega)$ is a strong solution. 
For the $H^{1/2}$ critical or subcritical range of $p$, i.e., $1 < p \leq \frac{n+1}{n-1}$, it is also easy to see that the set of weak solutions for \eqref{eq-main} is the same as the set of critical points of the functional
\[
I_m(u) := \int_{\Omega}\frac12|(-\Delta +m^2)^{1/4}u|^2-\frac{m}{2}u^2+\frac{1}{p+1}|u|^{p+1}\,dx,
\]
which is continuously differentiable on $H^{1/2}_0(\Omega)$.
A weak solution $u$ of \eqref{eq-main} is said to be \textit{of least energy} if $u$ satisfies
\[
I_m(u) = \inf\{I_m(v) ~|~  v \not\equiv 0,\, I_m'(v) = 0\}.
\]

\subsection{Localization}
One ingredient for Theorem \ref{thm-main-1} is to change the original problem \eqref{eq-main} to an equivalent problem which contains only local differential operators. 
This kind of technique was introduced by Caffarelli-Silvestre \cite{CS} for the localization of the fractional Laplacians $(-\Delta)^s$ on the whole domain $\R^n$.
The localization for the spectral fractional Laplacian on a bounded domain was given by Cabre-Tan \cite{CT}. 
The localization of the problem turns out to be powerful when we use various useful tools, such as the Pohozaev type identities, the Kelvin transform, and the moving plane method.
In this subsection, we briefly review it. 

We consider a cylinder $\mathcal{C}:= \Omega \times \R_+$. The symbol $\partial_L\mathcal{C}$ denotes the lateral boundary of $\mathcal{C}$, defined by $\{(x,t) \in \mathcal{C} ~|~ x \in \partial\Omega,\, t > 0\}$. 
We shall work in the space 
\begin{equation*}
H_{0,L}^1 (\mathcal{C}) = \{ U \in H^1 (\mathcal{C}) ~|~ U(x,t) = 0 \quad \textrm{on}~\partial_{L} \mathcal{C}\},
\end{equation*}
which is equipped with the norm
\[
\|U\| := \left(\int_{\mathcal{C}}|\nabla U(x,t)|^2\,dxdt\right)^{1/2}.
\]

For a given $u \in H_0^{1/2}(\Omega)$, consider a unique function $U \in H_{0,L}^1 (\mathcal{C})$ satisfying
\begin{equation}\label{extension}
\left\{\begin{array}{lll}
(-\Delta + m^2) U(x,t) =0 &\quad \textrm{in}~\mathcal{C},
\\
U(x,t) = 0&\quad \textrm{on}~\partial_L \mathcal{C},
\\
U(x,0) = u(x) &\quad \textrm{on}~\Omega \times \{0\}.
\end{array}
\right.
\end{equation}
Then, as in \cite[Proposition 2.2]{CT} and \cite[Theorem 5]{A}, the following equality holds in distributional sense:
\begin{equation*}
\sqrt{-\Delta +m^2} \,u (x) = \frac{\partial}{\partial \nu} U (x,0)\quad x \in \Omega.
\end{equation*}
Thus, our study of the problem \eqref{eq-main} can be transformed into the study of the following localized problem:
\begin{equation}\label{eq-lm}
\left\{\begin{array}{ll} 
(-\Delta + m^2) U(x,t)= 0&\quad \textrm{in}~\mathcal{C},
\\
U(x,t) = 0&\quad \textrm{on}~\partial_{L}\mathcal{C},
\\
\partial_{\nu} U(x,0) = mU(x,0) + |U(x,0)|^{p-1}U(x,0) &\quad \textrm{on}~\Omega \times \{0\},
\end{array}
\right.
\end{equation}
which is the Euler-Lagrange equation of the functional
\begin{equation}\label{functional-ext}
I_{e,m}(U) = \frac12\int_{\mathcal{C}}|\nabla U(x,t)|^2+m^2U(x,t)^2\,dxdt -\int_{\Omega}\frac{m}{2}U(x,0)^2+\frac{1}{p+1}|U(x,0)|^{p+1}\,dx
\end{equation}
defined on $H^1_{0,L}(\mathcal{C})$.
In other words, we may find a critical point of $I_{e,m}$ in order to find a weak solution of \eqref{eq-main}.
In addition, we note that
\[
I_m(u) = I_{e,m}(U),
\]
whenever $U$ is an extension of $u$ given by \eqref{extension}. 
Therefore, the least energy critical point of $I_{e,m}$ corresponds to the least energy critical point of $I_m$.

\subsection{Regularity of weak solutions} Here, we are concerned with the regularity property related to the problem \eqref{eq-lm}. 
Indeed, the required regularity can be demonstrated through minor modifications to the arguments in the proofs of Proposition 3.1 and Theorem 5.2 in \cite{CT}, 
where the corresponding results were obtained for the case that $m=0$.
\begin{lem}\label{lem-2-1}
Assume that $v \in H_{0,L}^1 (\mathcal{C})$ is a weak solution of
\begin{equation*}
\left\{\begin{array}{rll}
-\Delta v + m^2 v &=0&\textrm{in}~\mathcal{C},
\\
v&=0&\textrm{on}~\partial \Omega \times [0,\infty),
\\
\partial_{\nu} v&=f&\textrm{on}~\Omega \times \{0\}
\end{array}
\right.
\end{equation*}
in the sense that 
\begin{equation}
\int_{\mathcal{C}} \nabla v \cdot \nabla \phi + m^2 v \,\phi\, dx dt = \int_{\Omega} f (x) \phi (x,0) dx\quad \forall ~\phi \in H_{0,L}^1 (\mathcal{C}).
\end{equation}
Then, we have the following regularity results. 
\begin{enumerate}
\item If $f \in L^{q}(\Omega)$ and $v \in L^q([0,R]\times \Omega)$ for some $q >n+1$ and $R > 0$, then $v \in C^{\alpha} ([0,R/2]\times\overline\Omega)$ for some $\alpha \in (0,1)$.
\item If $v \in C^\alpha([0,R]\times \overline\Omega)$, $f \in C^{\alpha} (\overline{\Omega})$, and $f|_{\partial \Omega} \equiv 0$ for some $R > 0$ and $\alpha \in (0,1)$, then $v \in C^{1,\alpha}([0,R/2]\times\overline\Omega)$.
\item If $v \in C^{1,\alpha}([0,R]\times \overline\Omega)$, $f \in C^{1,\alpha}(\overline{\Omega})$, and $f|_{\partial \Omega} \equiv 0$ for some $R > 0$ and $\alpha \in (0,1)$, then $v\in C^{2,\alpha}([0,R/2]\times\overline\Omega)$.
\end{enumerate}
\end{lem}
\begin{proof}
The proof of this lemma follows along exactly the same lines as that of \cite[Proposition 3.1]{CT}. The fundamental idea in \cite{CT} is to make use of the function 
\begin{equation*}
w(x,t) = \int_0^t v(x,s) ds\quad \textrm{for}~ (x,t) \in \mathcal{C}.
\end{equation*}
This function satisfies 
\[\partial_t (-\Delta + m^2) w = (-\Delta +m^2) \partial_t w =(-\Delta +m^2) v =0,\]
which means that $(-\Delta +m^2) w$ is independent of $y$. Note that $(-\Delta + m^2) w (x,0) = - \partial_t v(x,0) = f(x)$ for $x \in \Omega$. Hence, $w$ satisfies
\begin{equation}\label{eq-2-1}
\left\{ \begin{aligned}
(-\Delta +m^2) w (x,t) &= f (x) &\textrm{in}~\mathcal{C},
\\
w&=0&\textrm{on}~\partial \mathcal{C},
\end{aligned}
\right.
\end{equation}
and if we extend $w$ by odd reflection to the whole cylinder $\Omega \times \mathbb{R}$ as
\begin{equation*}
w_{odd} (x,t) = \left\{ \begin{array}{ll} w(x,t) &\textrm{for}~ t \geq 0,
\\
-w(x,-t)&\textrm{for}~t \leq 0,
\end{array}
\right.
\end{equation*}
then the function $w_{odd}$ satisfies the same problem \eqref{eq-2-1} on $\Omega \times (-\infty, \infty)$. At this stage, we may use the classical elliptic regularity of $w_{odd}$, as in \cite{CT}, which implies the regularity properties of the function $v$.   We refer the reader to \cite{CT} for further details.
\end{proof}

From Lemma \ref{lem-2-1}, the following corollary easily follows. 
\begin{cor}\label{regularity-bdd}
For $p > 1$, any $L^\infty$ bounded weak solution of \eqref{eq-main} belongs to $C^{2,\alpha}(\overline{\Omega})$ for some $\alpha \in (0, 1)$. 
\end{cor}
\begin{proof}
Note that for any weak solution $u$ of \eqref{eq-main}, there exists a unique function $U$ such that $U(x,0) = u(x)$ satisfying
\begin{equation*}
\left\{\begin{array}{rll}
-\Delta U + m^2 U &=0&\textrm{in}~\mathcal{C},
\\
U&=0&\textrm{on}~\partial \Omega \times [0,\infty),
\\
\partial_{\nu} U&= |u|^{p-1}u&\textrm{on}~\Omega \times \{0\}.
\end{array}
\right.
\end{equation*}
From the maximum principle, one has $\sup_{\Omega \times [0,\infty)}|U| \leq \|u\|_{L^\infty(\Omega)}$ so Lemma \ref{lem-2-1} applies.
\end{proof}

The next corollary, an another application of Lemma \ref{lem-2-1} says every weak solution to \eqref{eq-main} is $C^{2,\alpha}(\overline\Omega)$ when $p$ is $H^{1/2}$ subcritical or critical.
\begin{cor}\label{regularity}
Let $p \in (1, \frac{n+1}{n-1}]$.
Then, any weak solution of \eqref{eq-lm} belongs to $C^{2,\alpha}(\overline{\mathcal{C}})$ for some $\alpha \in (0,1)$. 
In particular, any weak solution of \eqref{eq-main} belongs to $C^{2,\alpha}(\overline\Omega)$ so that it is a strong solution.
\end{cor}
\begin{proof}
In the spirit of Brezis-Kato estimates, we first need to show that $U \in L^q (\mathcal{C})$
and $U(\cdot,0) \in L^q(\Omega)$ for all $q > 1$. (See, for example, Theorem 5.2 in \cite{CT}.) Indeed, if we define $U_T = \min \{|U|, T\}$ for $T >1$ and test $U U_T^{2\beta}$ for \eqref{eq-lm} for $\beta \geq 0$, then we find that
\begin{equation*}
\int_{\mathcal{C}} \nabla U \nabla (U U_T^{2\beta}) dx dt + m^2 \int_{\mathcal{C}} U^2 U_T^{2\beta} dxdt = \int_{\Omega \times\{0\}} |U(x,0)|^{p+1} |U_T (x,0)|^{2\beta} dxdt.
\end{equation*}
This implies that
\begin{equation*}
\int_{\mathcal{C}} \nabla U \nabla (U U_T^{2\beta}) dx dt \leq \int_{\Omega \times\{0\}} |U(x,0)|^{p+1} |U_T (x,0)|^{2\beta} dxdt,
\end{equation*}
Then, all of the arguments of \cite[Theorem 5.2]{CT} can be applied in exactly the same manner, to obtain that
\begin{equation}\label{eq-2-5}
\int_{\mathcal{C}}|\nabla(UU_T^\beta)|^2
\leq C(\beta+1)\int_{\Omega \times\{0\}} |U(x,0)|^{p+1} |U_T (x,0)|^{2\beta} dxdt
\end{equation}
and consequently $U(\cdot,0) \in L^q(\Omega)$ for any $q > 1$.  Then, we may apply the Poincare inequality \eqref{ineq-poincare} to the left hand side of \eqref{eq-2-5} and take $T \rightarrow \infty$, to prove that $U \in L^q (\mathcal{C})$ for any $q>1$.
Finally, we apply Lemma \ref{lem-2-1} repeatedly to show that $U \in C^{2,\alpha}(\overline{\mathcal{C}})$ for some $\alpha \in (0,1)$.
\end{proof}
Note that Corollary \ref{regularity} gives a proof of regularity part of Theorem \ref{thm-main-1}.

\subsection{Embeddings and entire problems}\label{subsec_sobolev_trace} In this subsection, we review the best Sobolev embedding, along with the related entire problems and maximizing functions.
Given any $\lambda > 0$ and $\xi \in \mathbb{R}^n$, let
\begin{equation}\label{bubble}
w_{\lambda, \xi} (x) = \mathfrak{c}_{n} \left( \frac{\lambda}{\lambda^2+|x-\xi|^2}\right)^{\frac{n-1}{2}} \quad \text{for } x \in \mathbb{R}^n,
\end{equation}
where
\begin{equation}\label{cns}
\mathfrak{c}_{n} = 2^{\frac{n-1}{2}} \left( \frac{\Gamma \left( \frac{n+1}{2}\right)}{\Gamma \left( \frac{n-1}{2}\right)}\right)^{\frac{n-1}{2}}.
\end{equation}
We recall the sharp fractional Sobolev inequality for $n>1$, 
\begin{equation}\label{eq-sh-sb}
\left( \int_{\mathbb{R}^n} |f(x)|^{\frac{2n}{n-1}} dx \right)^{\frac{n-1}{2n}} \leq \mathcal{S}_{n} \left( \int_{\mathbb{R}^n} |(-\Delta)^{1/4} f(x)|^2 dx \right)^{1 \over 2} \quad \text{for any } f \in H^{1/2}(\mathbb{R}^n)
\end{equation}
with the best constant
\begin{equation}\label{ans}
\mathcal{S}_{n} =2^{-\frac{1}{2}} \pi^{-1/4} \biggl[\frac{\Gamma \left(\frac{n-1}{2}\right)}{\Gamma \left( \frac{n+1}{2}\right)}\biggr]^{1 \over 2} \biggl[ \frac{\Gamma(n)}{\Gamma(n/2)}\biggr]^{1 \over 2n}.
\end{equation}
The equality holds if and only if $u(x) = c w_{\lambda,\xi}(x)$ for any $c > 0,\ \lambda>0$ and $\xi \in \mathbb{R}^n$ (refer to \cite{CL, FL, L2}).
Furthermore, it was shown in \cite{CLO, L1, L3} that $\{ w_{\lambda, \xi}(x): \lambda>0, \xi \in \mathbb{R}^n \}$ is the set of all solutions for the problem
\begin{equation}\label{entire_nonlocal}
(-\Delta)^{1/2} u = u^{(n+1)/(n-1)},\quad u > 0 \quad \text{in } \mathbb{R}^n\quad \textrm{and}\quad \lim_{|x|\rightarrow \infty} u(x) = 0.
\end{equation}
We use $W_{\lambda,\xi} \in \mathcal{D}^1(\mathbb{R}^{n+1}_+)$ to denote the (unique) harmonic extension of $w_{\lambda,\xi}$, so that $W_{\lambda, \xi}$ solves
\begin{equation}\label{wlyxt}
\left\{\begin{aligned}
-\Delta W_{\lambda,\xi}(x,t) &= 0 &\quad \text{in~} \mathbb{R}^{n+1}_+,\\
W_{\lambda,\xi}(x,0) &= w_{\lambda,\xi}(x) &\quad \text{for } x \in \mathbb{R}^n.
\end{aligned}\right.
\end{equation}
It is easy to check that
\begin{equation}\label{eq-10-14}
W_{\lambda, \xi} (x,t) = \mathfrak{c}_n \left( \frac{\lambda}{|x-\xi|^2+(t+\lambda)^2} \right)^{\frac{n-1}{2}}\quad \textrm{for}~(x,t) \in \mathbb{R}^n \times [0,\infty),
\end{equation}
and for any $U \in D^1(\R^{n+1}_+)$, one has the trace Sobolev inequality
\begin{equation}\label{eq-sharp-trace}
\left( \int_{\mathbb{R}^n} |U(x,0)|^{\frac{2n}{n-1}} dx \right)^{\frac{n-1}{2n}} 
\leq \mathcal{S}_{n} \left( \int_0^{\infty}\!\!\!\int_{\mathbb{R}^n}  |\nabla U(x,t)|^2 dx dt \right)^{1 \over 2},
\end{equation}
where the equality is attained by some function $U \in \mathcal{D}^1(\mathbb{R}^{n+1}_+)$ if and only if $U(x,t) = c W_{\lambda,\xi}(x,t)$ for any $c > 0,\ \lambda>0$ and $\xi \in \mathbb{R}^n$. We remark that
\[
W_{\lambda,\xi}(x,t) = \lambda^{-\frac{n-1}{2}}W_{1,0}\left(\frac{x-\xi}{\lambda},\,\frac{t}{\lambda}\right).
\]
After this section we simply denote $w_{1,0}$ and $W_{1,0}$ by $w_1$ and $W_1$, respectively.

\subsection{Trace inequalities}
Here, we collect some useful trace inequalities which will be invoked frequently in the rest of the paper. 
\begin{prop}\label{lem-ce-1}
There exists a constant $C > 0$ depending only on $n, m$ and $\Omega$ such that the following inequality holds for any $U \in H_{0,L}^{1}(\mathcal{C})$:
\begin{equation*}
\left(\int_{\mathcal{C}}|\nabla U(x,t)|^2 dx dt + m^2 \int_{\mathcal{C}} |U(x,t)|^2 dx dt - m \int_{\Omega} U(x,0)^2 dx\right) \geq C \int_{\mathcal{C}} |\nabla U(x,t)|^2 dx dt.
\end{equation*}
\end{prop}
\begin{proof}
For given $A >0$, we apply integration by parts and Young's inequality to get that
\begin{equation}\label{eq-2-30}
\begin{split}
\int_{\Omega} mU(x,0)^2 dx &=-m\int_0^\infty \left\{\frac{d}{dt}\int_\Omega U(x,t)^2 dx\right\}dt\\
&=- 2m \int_{\Omega}\left( \int_{0}^{\infty} (\partial_t U) (x,t) U(x,t) dt\right) dx
\\
& \leq A m^2 \int_{\mathcal{C}} U(x,t)^2 dx dt + \frac{1}{A} \int_{\mathcal{C}} |\partial_t U(x,t)|^2 dx dt.
\end{split}
\end{equation}
In addition, by applying the Poincare inequality in the level $\{(x,h): x \in \Omega\}$ for each $h \geq 0$, we obtain the estimate
\begin{equation}\label{ineq-poincare}
\int_{\mathcal{C}} |U(x,t)|^2 dx dt \leq C(\Omega) \int_{\mathcal{C}} |\nabla_x U(x,t)|^2 dx dt.
\end{equation}
Now, we apply this estimate to find that
\begin{equation*}
\begin{split}
&\int_{\mathcal{C}} |\nabla U(x,t)|^2 dx dt + m^2 \int_{\mathcal{C}} |U(x,t)|^2 dxdt - m \int_{\Omega} U(x,0)^2 dx
\\
&\quad \quad\quad\geq \frac{1}{2}\int_{\mathcal{C}} |\partial_x U(x,t)|^2 dx dt  + \int_{\mathcal{C}} |\partial_t U(x,t)|^2 dx dt
\\
&\qquad\quad\quad  + \left(m^2 + \frac{C(\Omega)}{2}\right) \int_{\mathcal{C}} |U(x,t)|^2 dx dt - m\int_{\Omega} U(x,0)^2 dx.
\end{split}
\end{equation*}
Then, applying the estimate \eqref{eq-2-30} here with $A= \left(\frac{1}{m^2 + C(\Omega)/2}\right)$, we get that
\begin{equation*}
\begin{split}
&\int_{\mathcal{C}} |\nabla U(x,t)|^2 dx dt + m^2 \int_{\mathcal{C}} |U(x,t)|^2 dxdt - m \int_{\Omega} U(x,0)^2 dx
\\
&\quad \quad\quad\geq \frac{1}{2}\int_{\mathcal{C}} |\partial_x U(x,t)|^2 dx dt  + \frac{C(\Omega)/2}{m^2  + C(\Omega)/2}\int_{\mathcal{C}} |\partial_t U(x,t)|^2 dx dt
\\
&\quad \quad\quad\geq C_1 \int_{\mathcal{C}} |\nabla U(x,t)|^2 dxdt, 
\end{split}
\end{equation*}
where we have set $C_1 = \min \left\{ \frac{1}{2}, \frac{C(\Omega)/2}{m^2 + C(\Omega)/2}\right\}$. 
Thus, the proof is complete.
\end{proof}

As a byproduct of the above proof, we also obtain the following useful trace inequality.
\begin{prop}
For any $U \in H^1_{0,L}(\mathcal{C})$, it holds that
\begin{equation}\label{ineq-trace}
\int_{\Omega} mU^2(x,0)\,dx \leq m^2 \int_{\mathcal{C}} U^2\,dxdt + \int_{\mathcal{C}} |\partial_t U|^2\,dxdt.
\end{equation}
\end{prop}

The following Sobolev-trace embedding is well known (for example, see \cite{CT}).
\begin{prop}\label{prop-Sobolev-embedding}
For any $U \in H^1_{0,L}(\mathcal{C})$, its trace $U(\cdot,0)$ is continuously embedded in $L^q(\Omega)$ for every $q \in (1, \frac{2n}{n-1})$, i.e.,
there exists a constant $C > 0$ depending only on $n, q$ and $\Omega$ such that
\[
\|U(\cdot,0)\|_{L^q(\Omega)} \leq C\|U\|_{H^1_{0,L}(\mathcal{C})}.
\]
Moreover, the embedding is compact for any $q \in (1, \frac{2n}{n-1})$.
\end{prop}

\section{Nonexistence result (Proof of Theorem \ref{thm-ne})}\label{sec-nonex}

In this section, we prove non-existence of a nontrivial solution to the Dirichlet problem \eqref{eq-main} in the $H^1$-critical and supercritical cases. For this purpose, we shall use the following Pohozaev type identity which is obtained by testing $x\cdot\nabla_x U$ to \eqref{eq-lm},
not $(x,t)\cdot\nabla U$. 
\begin{prop}
Let $U \in H^1_{0, L}(\mathcal{C}) \cap L^{\infty}(\overline{\mathcal{C}})$ be a weak solution of \eqref{eq-lm}. Then one has
\begin{equation}\label{eq-ne-7}
\begin{split}
&\left(\frac{n-2}{2}\right) \int_{\mathcal{C}} |\nabla_x U|^2 dxdt+\frac{n}{2} \int_{\mathcal{C}} |\partial_t U|^2 dxdt
+\frac{nm^2}{2}\int_{\mathcal{C}} U^2 dx dt \\
&\qquad\qquad -\frac{n}{p+1}\int_{\Omega} U^{p+1} dx-\frac{nm}{2} \int_{\Omega} U^2 dx
+\frac{1}{2} \int_{\partial_L \mathcal{C}} (\nu_x \cdot x)\left( \frac{\partial U}{\partial \nu_x}\right)^2 dS = 0.
\end{split}
\end{equation}  
\end{prop}

\begin{proof}
By Corollary \ref{regularity-bdd}, we see that $U \in C^2(\mathcal{C})$. Multiplying the localized problem \eqref{eq-lm} by $x \cdot \nabla_x U$ and then integrating out, we get
\begin{equation}\label{eq-ne-1}
m^2 \int_{\mathcal{C}} U \cdot (x \cdot\nabla_x U) dxdt +\int_{\mathcal{C}} (-\Delta U) \cdot (x \cdot \nabla_x U) dx dt = 0	
\end{equation}
Integrating by parts, the first integral becomes
\begin{equation}\label{eq-ne-2}
m^2 \int_{\mathcal{C}} U \cdot (x \cdot\nabla_x U) dxdt
=\frac{m^2}{2}\int_{\mathcal{C}} x \cdot \nabla_x (U^2) dx dt = - \frac{m^2}{2}\int_{\mathcal{C}}n U^2 dx dt,
\end{equation}
while the second integral becomes 
\begin{equation}\label{eq-ne-3}
\begin{split}
\int_{\mathcal{C}} (-\Delta U) (x \cdot \nabla_x U) dxdt 
& =\int_{\partial \mathcal{C}} -\frac{\partial U}{\partial \nu} ( x \cdot \nabla_x U) dx dt 
+ \int_{\mathcal{C}} \nabla U \cdot \nabla ( x \cdot \nabla_x U) dxdt \\
& = - \int_{\Omega} (U^p + mU) (x \cdot \nabla_x U) dx 
-\int_{\partial_L \mathcal{C}} (\nu_x\cdot x) \left| \frac{\partial U}{\partial \nu_x}\right|^2 dS\\
&\quad + \int_{\mathcal{C}} \nabla U \cdot \nabla ( x \cdot \nabla_x U) dxdt\\
&=: I+II+III.
\end{split}
\end{equation}
By integration by parts as in \eqref{eq-ne-2}, one has
\begin{equation}\label{eq-ne-4}
I= -\frac{n}{p+1} \int_{\Omega} U^{p+1} - \frac{nm}{2} \int_{\Omega} U^2 dx.
\end{equation}
Next, we split $III$ as follows.
\begin{equation*}
III = \int_{\mathcal{C}}(\nabla_x U) \cdot \nabla_x ( x \cdot \nabla_x U) dx dt + \int_{\mathcal{C}} \partial_t U \, \partial_t ( x \cdot \nabla_x U) dx dt=:III_1+III_2.
\end{equation*}
By integration by parts again, we write
\begin{equation}\label{eq-ne-6}
\begin{split}
III_1&= \sum_{i,j=1}^n \int_{\mathcal{C}}\partial_i U \partial_i (x_j \partial_j U) dx dt
= \sum_{i,j=1}^n  \int_{\mathcal{C}} \partial_i U ( \delta_{ij} \partial_j U + x_j \partial_{ij} U) dx dt \\
& = \int_{\mathcal{C}} |\nabla_x U|^2 dx dt + \int_{\mathcal{C}} \frac{x}{2} \nabla_x |\nabla_x U|^2 dx dt \\
& = \left(1- \frac{n}{2}\right) \int_{\mathcal{C}} |\nabla_x U|^2 dxdt + \int_{\partial_L \mathcal{C}} \frac{1}{2} (\nu_x \cdot x) \left| \frac{\partial U}{\partial \nu_x}\right|^2 dS.
\end{split}
\end{equation}
On the other hand, since $U=0$ on $\partial_{L}\mathcal{C}$ implies $\partial_tU = 0$ on $\partial_L\mathcal{C}$, we have
\begin{equation}\label{eq-ne-5}
\begin{split}
III_2&= \frac{1}{2} \int_{\mathcal{C}} x\cdot \nabla_x  (\partial_t U)^2 dx dt \\
&= \frac{1}{2} \int_{\partial_L \mathcal{C}} (\nu_x \cdot x) (\partial_t U)^2 dxdt 
-\frac{n}{2} \int_{\mathcal{C}} (\partial_t U)^2 dxdt \\
&= -\frac{n}{2} \int_{\mathcal{C}} (\partial_t U)^2 dxdt.
\end{split}
\end{equation}
Inserting the above identities \eqref{eq-ne-4}, \eqref{eq-ne-6} and \eqref{eq-ne-5} into \eqref{eq-ne-3}, we get
\begin{equation*}
\begin{split}
\int_{\mathcal{C}} (-\Delta U) (x \cdot \nabla_x U) dxdt
& = \frac{n}{p+1} \int_{\Omega} U^{p+1} dx + \frac{mn}{2} \int_{\Omega} U^2 dx
-\int_{\partial_L \mathcal{C}} \frac12(\nu_x \cdot x) \left| \frac{\partial U}{\partial \nu_x}\right|^2 dS \\
&\quad - \frac{n}{2} \int_{\mathcal{C}} |\partial_t U|^2 dx dt + \left(1- \frac{n}{2}\right) \int_{\mathcal{C}} |\nabla_x U|^2 dxdt.
\end{split}
\end{equation*}
Therefore, we conclude that 
\begin{equation*}
\begin{split}
&\left( \frac{n}{2} - 1 \right) \int_{\mathcal{C}} |\nabla_x U|^2 dxdt+\frac{n}{2} \int_{\mathcal{C}} |\partial_t U|^2 dxdt
+\frac{nm^2}{2}\int_{\mathcal{C}} U^2 dx dt \\
&\qquad\qquad -\frac{n}{p+1}\int_{\Omega} U^{p+1} dx-\frac{mn}{2} \int_{\Omega} U^2 dx
+\frac{1}{2} \int_{\partial_L \mathcal{C}} (\nu_x \cdot x)\left| \frac{\partial U}{\partial \nu_x}\right|^2 dS = 0,
\end{split}
\end{equation*}
which is \eqref{eq-ne-7}.
\end{proof}

Theorem \ref{thm-ne} follows from combining the Pohozaev identity \eqref{eq-ne-7}, the Nehari identity and the trace inequality. 
\begin{proof}[Proof of Theorem \ref{thm-ne}]
Let $u$ be a bounded weak solution of \eqref{eq-main} with $p \geq (n+2)/(n-2)$. 
Let $U \in H^1_{0,L}(\mathcal{C})$ be a extension of $u$ given by \eqref{extension}.
Then $U$ satisfies \eqref{eq-lm} and $\|U\|_{L^{\infty}(\mathcal{C})} \leq \|U(x,0)\|_{L^{\infty}(\Omega)} <\infty$ by the maximum principle.
Multiplying \eqref{eq-lm} by $U$ and taking an integration by parts, we obtain
\begin{equation}\label{eq-ne-8}
\int_{\Omega} mU^2 + U^{p+1} dx = m^2 \int_{\mathcal{C}} U^2 dx + \int_{\mathcal{C}} |\nabla U|^2\,dxdt.
\end{equation}
Multiplying \eqref{eq-ne-8} by $-\frac{n}{p+1}$ and summing up with \eqref{eq-ne-7}, we get
\begin{equation*}
\begin{split}
\left(\frac{n}{2}-\frac{n}{p+1}\right)\int_{\Omega} mU^2\,dx & =\left(\frac{n}{2}-\frac{n}{p+1}\right)m^2\int_{\mathcal{C}}U^2\,dxdt 
+\frac{1}{2}\int_{\partial_L \mathcal{C}}(\nu_x \cdot x)\left|\frac{\partial U}{\partial \nu_x}\right|^2 dS_z \\
&\quad + \left( \frac{n}{2} - \frac{n}{p+1}\right) \int_{\mathcal{C}}|\partial_t U|^2 dxdt 
+ \left(\frac{n}{2}-1-\frac{n}{p+1}\right)\int_{\mathcal{C}}|\nabla_x U|^2\,dxdt.
\end{split}
\end{equation*}
Since $p \geq\frac{n+2}{n-2}$, we have
\begin{equation*}
\left( \frac{n}{2} - 1 - \frac{n}{p+1}\right) \int_{\mathcal{C}} |\nabla_x U|^2 dx dt \geq 0.
\end{equation*}
Recall the trace inequality \eqref{ineq-trace}
\begin{equation*}
m^2 \int_{\mathcal{C}} U^2 dx dt + \int_{\mathcal{C}} |\partial_t U|^2 dx dt \geq \int_{\Omega} mU^2(x,0) dx,
\end{equation*}
from which we conclude that
\begin{equation*}
\frac{1}{2} \int_{\partial_L \mathcal{C}} (\nu_x \cdot x) \left| \frac{\partial U}{\partial \nu_x}\right|^2 dS_z  = 0.
\end{equation*}
Then we deduce $U \equiv 0$ from the unique continuation property. This completes the proof.
\end{proof}

\section{Existence in $H^{1/2}$ subcritical case}\label{sec-ex-sub}
In this section, we shall construct a positive least energy solution of \eqref{eq-lm} for all $m > 0$ and $p\in (1, (n+1)/(n-1))$.
We note that once we find a least energy solution of \eqref{eq-lm},
the sign-definiteness of it follows from a standard argument. For example, we refer to Theorem 1.2 in \cite{ChS}. Thus, we may focus ourselves on construction of a least energy solution of \eqref{eq-lm}. 

Let $p \in (1,\, (n+1)/(n-1))$. We recall that
\[
I_{e,m}(U) = \frac{1}{2} \int_{\mathcal{C}} |\nabla U|^2 +m^2 U^2 dx dt - \frac{m}{2} \int_{\Omega}U^2 (x,0) dx - \frac{1}{p+1} \int_{\Omega}|U(x,0)|^{p+1}dx.
\]
As is mentioned in Section \ref{prelim}, we may find a least energy critical point of $I_{e,m}$ to find a least energy critical point of $I_{e,m}$.
We shall search for a minimizer of $I_{e,m}$ on the Nehari manifold
\begin{equation*}
\mathcal{N}_{e,m} := \{V \in H^1_{0,L}(\mathcal{C}) ~|~ J_{e,m}(V) = 0,\, V \not\equiv 0 \},
\end{equation*}
where $J_{e,m}(V) := I'_{e,m}(V)V$, i.e.,
\begin{equation}\label{eq-je}
J_{e,m}(V) =\int_{\mathcal{C}} |\nabla V (x,t)|^2 + m^2  V(x,t)^2 dxdt  -m\int_{\Omega} V(x,0)^2 dx -  \int_{\Omega} |V(x,0)|^{p+1} dx.
\end{equation}

\begin{lem}\label{nonmepty}
$\mathcal{N}_{e,m}$ is nonempty.
\end{lem}
\begin{proof}
Choose any nonzero $V \in H^1_{0,L}(\mathcal{C})$. Then as a function of $t$, one can see
\[
I_{e,m}(tV) = \frac{t^2}{2}\left(\int_{\mathcal{C}}|\nabla V|^2+m^2 V^2\,dxdy
-m\int_{\Omega}V(x,0)^2\,dx\right)
-\frac{t^{p+1}}{p+1}\int_{\Omega}|V(x,0)|^{p+1}\,dx
\]
attains a unique local maximum at some $t_0 \in (0,\,\infty)$.
Differentiating with respect to $t$, we get $I_{e,m}'(t_0 V)V = 0$, which says $t_0 V \in \mathcal{N}_{e,m}$.
\end{proof}

\begin{lem}\label{minimizer}
There exists a minimizer of $I_{e,m}$ subject to $\mathcal{N}_{e,m}$, i.e.,
there exists a function $U \in \mathcal{N}_{e,m}$ such that $I_{e,m}(U) = M_{e,m}$.

\end{lem}
\begin{proof}

Take a minimizing sequence $\{V_j\}$ of $I_{e,m}$ subject to $\mathcal{N}_{e,m}$.
Then one has
\[
\left(\frac{1}{2}-\frac{1}{p+1}\right)
\left(\int_{\mathcal{C}}|\nabla V_j|^2+m^2 V_j^2\,dxdt-m\int_{\Omega}V_j(x,0)^2\,dx\right)
 = I_{e,m}(V_j) < C,
\]
from which and Proposition \ref{lem-ce-1} we deduce  $\|V_j\|$ is bounded for $j$.
Then there exists some $V_0 \in H^1_{0,L}(\mathcal{C})$
such that, after extracting a subsequence, $V_j \rightharpoonup V_0$ weakly in $H^1_{0,L}(\mathcal{C})$
and $V_j(x,0) \to V_0(x,0)$ strongly in $L^{p+1}(\Omega)$ by Proposition \ref{prop-Sobolev-embedding}.
The weakly lower semi-continuity of the functional $I_{e,m}$ and $J_{e,m}$, again came from Proposition \ref{prop-Sobolev-embedding}, says
$I_{e,m}(V_0) \leq M_{e,m}$ and $J_{e,m}(V_0) \leq 0$.
To show that $V_0\neq 0$ by contradiction, suppose that $V_0 \equiv 0$. 
Then one has $V_j(x,0) \to 0$ in $L^{p+1}(\Omega)$
but since $J_{e,m}(V_j) = 0$,  we have
\begin{equation*}
\begin{split}
\int_{\Omega} |V_j (x,0)|^{p+1}\,dx&
=\int_{\mathcal{C}} |\nabla V_j(x,t)|^2 + m^2V_j(x,t)^2\,dxdt -m\int_{\Omega} V_j(x,0)^2\,dx
\\
& \geq  C\int_{\mathcal{C}} |\nabla V_j(x,t)|^2\,dxdt
\\
& \geq C \| V_j (\cdot, 0)\|_{L^{p+1}(\Omega)}^2,
\end{split}
\end{equation*}
from the trace inequality \eqref{ineq-trace} and trace Sobolev inequality.
This shows that $\|V_j(\cdot,0)\|_{L^p(\mathbb{R}^n)}$ is bounded below from $C^{\frac{1}{p-1}}$
and consequently $V_0 \not\equiv 0$.
Then it is easy to see there is $t_0 \in (0, 1]$ such that
$I_{e,m}(t_0V_0) \leq M_{e,m}$ and $J_{e,m}(t_0V_0) = 0$.
This completes the proof of Lemma \ref{minimizer}.
\end{proof}

\begin{proof}[Proof of Theorem \ref{thm-main-1} for the subcritical case]
Let $U$ be the minimizer obtained in Lemma \ref{minimizer}.
Since $U \in \mathcal{N}_{e,m}$,
\[
\begin{aligned}
J_{e,m}'(U)U &=  2\left(\int_{\mathcal{C}}|\nabla U|^2+m^2 U^2\,dxdt
-m\int_{\Omega}U (x,0)^2\,dx\right) -(p+1) \int_{\Omega} |U(x,0)|^{p+1}\,dx \\
&= (1-p)\int_{\Omega} |U(x,0)|^{p+1}\,dx \neq 0
\end{aligned}
\]
so $J_{e,m}'(U) \neq 0$. Then the Lagrange multiplier rule applies to see that for some $\lambda \in \R$,
\begin{equation}\label{LM}
I_{e,m}'(U) = \lambda J_{e,m}'(U).
\end{equation}
By testing $U$ to \eqref{LM}, we see $\lambda = 0$ so that $U$ is a nontrivial solution of \eqref{eq-lm}.
We finally see that $U$ is a least energy solution of \eqref{eq-lm} since, for every nontrivial solution $V$ of \eqref{eq-lm}, 
one must have $V \in \mathcal{N}_{e,m}$ by testing $V$ to the equation \eqref{eq-lm}.
\end{proof}

\section{Existence in $H^{1/2}$ critical case}\label{sec-ex-crit}
In this section, we prove Theorem \ref{thm-main-1} in the critical case $p =\frac{n+1}{n-1}$ by finding a critical point of the functional
\begin{equation*}
I_{e,m} (U) = \frac{1}{2} \left[ \int_{\mathcal{C}} |\nabla U(x,t)|^2 + m^2 U(x,t)^2\,dxdt 
-m\int_{\Omega}U(x,0)^2\, dx \right] - \frac{n-1}{2n} \int_{\Omega}|U(x,0)|^{\frac{2n}{n-1}}\,dx.
\end{equation*}
Due to the loss of compactness of the embedding $H^1_{0,L}(\mathcal{C})\big|_{\Omega} \hookrightarrow L^{2n/(n-1)}(\Omega)$, the first step we have to do would be characterizing the levels of $I_{e,m}$ at which the Palais-Smale condition holds.  
\begin{lem}\label{lem-ce-3} 
The Palais-Smale condition holds for $I_{e,m}$ at any level $B < \frac{1}{2n}\mathcal{S}_n^{-2n}$.
\end{lem}
\begin{proof}
To verify the lemma, assume that $\{U_k\}_{k=1}^\infty$ is a sequence in $H_{0,L}^{1}(\mathcal{C})$ such that
\begin{equation}\label{eq-5-2}
\lim_{k\to\infty}I_{e,m}'(U_k) = 0 \quad \textrm{and}\quad \lim_{k \rightarrow \infty} I_{e,m} (U_k) = B
\end{equation}
for some $B < \frac{1}{2n}\mathcal{S}_n^{-2n}$.
We need to show that $\{U_k\}_{k=1}^\infty$ converges in $H_{0,L}^{1} (\mathcal{C})$ up to a subsequence. 
We observe from \eqref{eq-5-2} that
\begin{equation*}
\sup_{k\in \mathbb{N}} \int_{\Omega}|U_k(x,0)|^{\frac{2n}{n-1}}\, dx < \infty\quad \textrm{and}\quad 
\sup_{k \in \mathbb{N}} \left(\int_{\mathcal{C}} |\nabla U_k|^2 + m^2 U_k^2\,dxdt-m\int_{\Omega}U_k(x,0)^2\,dx\right) < \infty.
\end{equation*}
Hence $\{U_k\}_{k=1}^\infty$ has a weakly convergent subsequence in $H_{0,L}^1 (\mathcal{C})$, and its limit $U \in H_{0,L}^1 (\mathcal{C})$ satisfies
\begin{equation}\label{eq-5-3}
\left\{\begin{array}{rll} (-\Delta +m^2) U &=0\quad &\textrm{in}~\mathcal{C},
\\
\partial_{\nu} U &= m U + U^{\frac{n+1}{n-1}}~&\textrm{on}~\Omega \times \{0\}.
\end{array}
\right.
\end{equation}
By the embedding properties, we know that  $U_k(\cdot,0) \rightarrow U(\cdot,0)$ in $L^2 (\Omega)$, $U_k \rightharpoonup U$ in $L^{\frac{2n}{n-1}}(\Omega)$ and $U_k \rightharpoonup U$ in $L^2 (\mathcal{C})$. Hence we have
\begin{equation}\label{eq-5-4}
\lim_{k \rightarrow \infty} \int_{\mathcal{C}} |\nabla (U_k - U)|^2+|U_k - U|^2 \,dxdt = \lim_{k \rightarrow \infty} \int_{\mathcal{C}} |U_k|^2 + |\nabla U_k|^2 dx dt - \int_{\mathcal{C}} U^2 + |\nabla U|^2 dx,
\end{equation}
and the Brezis-Lieb lemma implies that
\begin{equation}\label{eq-5-5}
\lim_{k \rightarrow \infty} \int_{\Omega}|U_k - U|^{\frac{2n}{n-1}} dx = \lim_{k \rightarrow \infty} \int_{\Omega} |U_k|^{\frac{2n}{n-1}} dx - \int_{\Omega} |U|^{\frac{2n}{n-1}} dx.
\end{equation}
From the fact that $\lim_{k \rightarrow \infty} I_{e,m}'(U_k)U_k =0$, we have
\begin{equation}\label{eq-5-6}
0 = \lim_{k \rightarrow \infty}\left( \int_{\mathcal{C}} |\nabla U_k|^2 + m^2|U_k|^2\,dxdt  -\int_{\Omega} m|U_k|^2\,dx -\int_{\Omega}|U_k|^{\frac{2n}{n-1}}\,dx\right)
\end{equation} 
and from \eqref{eq-5-3},
\begin{equation}\label{eq-5-7}
0 = \int_{\mathcal{C}} |\nabla U|^2 + m^2|U|^2\,dxdt -\int_{\Omega} m|U|^2\,dx -\int_{\Omega}|U|^{\frac{2n}{n-1}}\,dx.
\end{equation}
Subtracting \eqref{eq-5-6} from \eqref{eq-5-7}, we see from \eqref{eq-5-4}, \eqref{eq-5-5} and $L^2$ convergence $U_k (\cdot,0) \rightarrow U(\cdot,0)$ in $\Omega$ that
\begin{equation*}
0 =  \lim_{k \rightarrow \infty} \left(\int_{\mathcal{C}} |\nabla(U_k -U)|^2 + m^2|U_k -U|^2 \,dxdt- \int_{\Omega} |U_k -U|^{\frac{2n}{n-1}}\,dx\right).
\end{equation*} 
After extracting a subsequence, we define a nonnegative value $J\geq 0$ by
\begin{equation}\label{eq-5-8}
J: = \lim_{k\rightarrow \infty}  \int_{\Omega} |U_k -U|^{\frac{2n}{n-1}} dx =\lim_{k \rightarrow \infty} \int_{\mathcal{C}} |\nabla(U_k -U)|^2 +m^2|U_k -U|^2\,dxdt.
\end{equation}
If $J =0$, we are done. Suppose that $J>0$. Then, using the trace Sobolev inequality \eqref{eq-sharp-trace}, we get
\begin{equation*}
\begin{split}
J =  \lim_{k \rightarrow \infty} \int_{\Omega} |U_k -U|^{\frac{2n}{n-1}}\,dx &\leq \mathcal{S}_n^{\frac{2n}{n-1}}\lim_{k\to\infty}\left( \int_{\mathcal{C}} |\nabla (U_k - U)|^2\,dxdt \right)^{\frac{n}{n-1}}
\\
&\leq \mathcal{S}_n^{\frac{2n}{n-1}}\lim_{k \rightarrow \infty} \left( \int_{\mathcal{C}}|\nabla (U_k - U)|^2 + m^2 |U_k -U|^2\,dxdt \right)^{\frac{n}{n-1}}
\\
&\leq \mathcal{S}_n^{\frac{2n}{n-1}}J^{\frac{n}{n-1}},
\end{split}
\end{equation*}
which shows $\mathcal{S}_n^{-2n} \leq J$. Combining this with \eqref{eq-5-8}, we get
\begin{equation}\label{eq-5-9}
\begin{split}
&\lim_{k \rightarrow \infty} \left[ \frac{1}{2} \int_{\mathcal{C}} |\nabla (U_k -U)|^2 +|U_k - U|^2\,dxdt - \frac{n-1}{2n}\int_{\Omega} |U_k -U|^{\frac{2n}{n-1}} dx \right]
\\
& = \lim_{k \rightarrow \infty} \frac{1}{2n} \int_{\Omega} |U_k- U|^{\frac{2n}{n-1}} dx  \geq \frac{1}{2n} \mathcal{S}_n^{-2n}.
\end{split}
\end{equation}
On the other hand, we may use again \eqref{eq-5-4}, \eqref{eq-5-5} and the fact that $U_k (\cdot,0) \rightarrow U(\cdot,0)$ in $L^2 (\Omega)$, to deduce that
\begin{equation}\label{eq-5-10}
\lim_{k \rightarrow \infty} I_{e,m} (U_k) = I_{e,m} (U) + \lim_{k \rightarrow \infty} \left[ \frac{1}{2} \int_{\Omega} |(U_k - U)|^2 + |\nabla (U_k -U)|^2\,dxdt - \frac{n-1}{2n}\int_{\Omega} |U_k -U|^{\frac{2n}{n-1}} dx \right]
\end{equation}
Noting that $I_{e,m} (U) \geq 0$ and inserting \eqref{eq-5-9} into \eqref{eq-5-10}, we obtain $B \geq \frac{1}{2n}\mathcal{S}_n^{-2n}$, 
which is a contradiction to the assumption that $B < \frac{1}{2n}\mathcal{S}_n^{-2n}$. Hence $J=0$ holds and thus the lemma is proved.
\end{proof}

We define the mountain pass level $L_{e,m}$ of $I_{e,m}$ by
\[
L_{e,m} := \inf_{\gamma \in \Gamma}\max_{t\in[0,1]}I_{e,m}(\gamma(t)),
\]
where
\[
\Gamma := \{\gamma \in C([0,1],\,H^1_{0,L}(\mathcal{C})) ~|~ \gamma(0) = 0,\, I_{e,m}(\gamma(1)) < 0\}.
\]

\begin{lem}\label{lem-ce-2} 
It holds that
\[
L_{e,m} < \frac{1}{2n}\mathcal{S}_n^{-2n}
\]
In particular, the Palais-Smale condition holds for $I_{e,m}$ at the mountain pass level $L_{e,m}$. 
\end{lem}
\begin{proof}
For any given $\Psi \in H_{0,L}^1 (\mathcal{C})$ we have
\begin{equation*}
I_{e,m} (t \Psi) = \frac{1}{2} t^2 \alpha - \frac{n-1}{2n} t^{\frac{2n}{n-1}} \beta,
\end{equation*}
where 
\begin{equation*}
\alpha = \int_{\mathcal{C}} |\nabla \Psi(x,t)|^2 + m^2 \Psi(x,t)^2\,dxdt - m\int_{\Omega}\Psi(x,0)^2\,dx  \quad \textrm{and}\quad \beta = \int_{\Omega} |\Psi(x,0)|^{\frac{2n}{n-1}} dx.
\end{equation*}
It is easy to see that the maximum value of the map $t \in (0,\infty) \mapsto I_{e,m} (t \Psi)$ is attained at $t_0 = (\alpha/\beta)^{\frac{n-1}{2}}$ and the value is equal to 
\begin{equation*}
\begin{split}
I_{e,m}  (t_0 \Psi) &= \frac{1}{2n} \alpha^n \beta^{-(n-1)}
\\
& = \frac{1}{2n} \frac{ \left( \int_{\mathcal{C}} |\nabla \Psi|^2 + m^2 \Psi^2\, dx dt -m\int_{\Omega}\Psi(x,0)^2 dx\right)^n}{ \left( \int_{\Omega} |\Psi(x,0)|^{\frac{2n}{n-1}} dx \right)^{n-1}}.
\end{split}
\end{equation*}
In order to finish the proof, we need to find a function $\Psi \in H_{0,L}^1 (\mathcal{C})$ such that 
\begin{equation*}
\max_{t>0} I_{e,m}(t \Psi) < \frac{1}{2n} \mathcal{S}_n^{-2n}.
\end{equation*}
We may suppose $0 \in \Omega$. We take $\Psi_{\lambda}(x,t) := \phi(x)W_{\lambda, 0}(x,t)$ where $W_{\lambda,0}$ is defined in \eqref{wlyxt} and $\phi \in C^\infty_c(\Omega)$ satisfying $\phi = 1$ on some ball $B_\rho(0) \subset \Omega$.
We are now going to estimate each term of $I_{e,m}(t \Psi_{\lambda})$ using the decay $W_{1,0}(z) \leq C/(1+|z|)^{n-1}$ for $z \in \mathbb{R}^{n+1}_{+}$.
First we estimate
\begin{equation*}
\begin{split}
\int_{\mathcal{C}} |\nabla \Psi_{\lambda}|^2 dx dt & = \int_{\mathcal{C}} |\nabla V_{\lambda}|^2 \phi^2\,dxdt +  \int_{\mathcal{C}}  2(\phi V_{\lambda})\nabla V_{\lambda}\cdot\nabla\phi + V_{\lambda}^2|\nabla \phi|^2\, dxdt
\\
&= \int_{\mathbb{R}^{n+1}_{+}}|\nabla V_{\lambda}|^2 dx dt + \int_{\mathbb{R}^{n+1}_{+}} (1- \phi^2 )|\nabla V_{\lambda}|^2 dx dt + O (\lambda^{n-1})
\\
& = \int_{\mathbb{R}^{n+1}_{+}} |\nabla V|^2 dx dt + O(\lambda^{n-1}),
\end{split}
\end{equation*}
and similarly, 
\begin{equation*}
\begin{split}
\int_{\Omega}|\Psi_{\lambda}(x,0)|^{\frac{2n}{n-1}} dx & = \int_{\mathbb{R}^{n}} |W_{\lambda,0}|^{\frac{2n}{n-1}}(x,0) \phi^{\frac{2n}{n-1}}(x,0) dx
\\
& = \int_{\mathbb{R}^n} |W_{\lambda,0}(x,0)|^{\frac{2n}{n-1}} dx + \int_{\mathbb{R}^n} |W_{\lambda}(x,0)|^{\frac{2n}{n-1}} ( \phi(x,0)^{\frac{2n}{n-1}} - 1)\,dx
\\
& = \int_{\mathbb{R}^n} |W_{1,0}(x,0)|^{\frac{2n}{n-1}} dx + O(\lambda^{n}).
\end{split}
\end{equation*}
Also we easily see that
\begin{equation*}
\int_{\mathcal{C}} m^2 \Psi_{\lambda}^2\,dxdt \leq \int_{\mathbb{R}^{n+1}_{+}} m^2 W_{\lambda,0}^2\,dxdt = C m^2 \lambda^{-(n-1)}\lambda^{n+1} \int_{\mathbb{R}^{n+1}_{+}} W_{1,0}^2\,dxdt = O(\lambda^{2}),
\end{equation*}
and
\begin{equation*}
\begin{split}
\int_{\Omega} m \Psi_{\lambda}^2 (x,0)\,dx &= m \lambda\int_{\mathbb{R}^n} W_{1,0}^2 (x,0)\,dx + m \int_{\mathbb{R}^n}(\phi^2-1)W_{\lambda,0}^2(x,0)\,dx
\\
&= m \lambda\int_{\mathbb{R}^n} W_{1,0}^2 (x,0)\,dx + O(\lambda^{n-1}).
\end{split}
\end{equation*}
Merging the estimates above, we deduce that for small $\lambda > 0$,
\begin{equation*}
\begin{split}
\sup_{t >0} E (t \psi_{\lambda})& = \frac{1}{2n} \frac{\left( \int_{\mathbb{R}^{n+1}_{+}} |\nabla W_{1,0}|^2\,dxdt 
-m\lambda\int_{\mathbb{R}^n} W_{1,0}^2(x,0)\, dx +O(\lambda^2)\right)^n}{\left( \int_{\mathbb{R}^n} |W_{1,0}(x,0)|^{\frac{2n}{n-1}} dx + O (\lambda^{n}) \right)^{n-1}}
\\
&< \frac{1}{2n} \frac{\left( \int_{\mathbb{R}^{n+1}_+} |\nabla W_{1,0}|^2 dxdt \right)^n}{\left( \int_{\mathbb{R}^n} |W_{1,0}(x,0)|^{\frac{2n}{n-1}} dx \right)^{n-1}} 
\\
&= \frac{1}{2n} \int_{\mathbb{R}^n} |W_{1,0}(x,0)|^{\frac{2n}{n-1}} dx = \frac{1}{2n}\mathcal{S}_n^{-2n},
\end{split}
\end{equation*}
where, in the last equalities, we have used that
\begin{equation*}
\left( \int_{\mathbb{R}^n} |W_{1,0}(x,0)|^{\frac{2n}{n-1}} dx \right)^{\frac{n-1}{2n}} = \mathcal{S}_n\left( \int_{\mathbb{R}^{n+1}_{+}} |\nabla W_{1,0}|^2 dx dt \right)^{1/2} 
= \mathcal{S}_n \left( \int_{\mathbb{R}^n} |W_{1,0}(x,0)|^{\frac{2n}{n-1}}dx \right)^{1/2},
\end{equation*}
which shows that
\begin{equation*}
\int_{\mathbb{R}^n} |W_{1,0}(x,0)|^{\frac{2n}{n-1}} dx = \mathcal{S}_n^{-2n}.
\end{equation*}
The proof is finished.
\end{proof}

We are ready to complete the whole the proof of Theorem \ref{thm-main-1}.
\begin{proof}[Proof of Theorem \ref{thm-main-1} for the critical case $p=\frac{n+1}{n-1}$]
To complete the proof, it remains to show $I_{e,m}$ enjoys the mountain pass geometry. In other words, we show that
\begin{itemize}
\item[(MP1)] there exists $r_0 > 0$ such that $I_{e,m}(V) \geq 0$ for any $v \in H^1_{0,L}(\mathcal{C})$ with $\|v\| \leq r_0$
and $\inf\{I_{e,m} (V) ~|~ V \in H^1_{0,L}(\mathcal{C}),\, \|V\| = r_0\} > 0$;
\item[(MP2)] there exists $V_0 \in H^1_{0,L}(\mathcal{C})$ such that $I_{e,m}(V_0) < 0$. 
\end{itemize}
By the trace Sobolev embedding and Proposition \ref{lem-ce-1}, wee see that
\begin{equation*}
I_{e,m} (U) \geq \frac{1}{C} \| U\|^2 - C  \| U\|^{\frac{2n}{n-1}}
\end{equation*}
so (MP1) holds. It is easy to see (MP2) also holds by observing the scaling map $t \mapsto I_{e,m}(tV)$.
Then the existence of a (PS) sequence of $I_{e,m}$ at the mountain pass level $L_{e,m}$ follows from the standard pseudo gradient flow argument (see \cite{W}) so that 
we have a critical point of $I_{e,m}$ with the level $L_{e,m}$.
Finally, we note that for any nontrivial solution $U$ to \eqref{eq-lm}, one must have
\[
\frac{d}{dt}\bigg|_{t=1}I_{e,m}(tU) = I_{e,m}'(U)U = 0,
\]
which shows that
\[
I_{e,m}(U) = \max_{t>0}I_{e,m}(tU).
\]
This means that $L_{e,m}$ is the least energy level. The proof is finished. 
\end{proof}

\section{Existence in $H^{1/2}$ supercritical case: Proof of Theorem \ref{thm-3} }\label{sec-ex-sup}



In this section, we show that the problem \eqref{eq-main} has a nontrivial solution even for the supercritical case $\frac{n+1}{n-1} < p < \frac{n+2}{n-2}$
under the assumption on $\Omega$ that the equation \eqref{eq-2} admits a non-degenerate solution $u_\infty \in H^1_0(\Omega)$, i.e., 
we assume there exists a nontrivial solution $u_{\infty} \in H^1_0(\Omega)$ to
\begin{equation}\label{eq-local}
-\Delta u=|u|^{p-1}u \quad\textup{in }\Omega, \quad u  = 0 \quad \text{on } \partial\Omega
\end{equation}
such that the linearized equation 
\[
-\Delta v = p|u_\infty|^{p-1}v \quad \text{in } \Omega, \quad v  = 0 \quad \text{on } \partial\Omega
\]
admits only trivial solution. 
By the scaling $u \mapsto (2m)^{\frac{1}{p-1}}u$, the equation \eqref{eq-main} is transformed to
\begin{equation}\label{eq-main-tran}
2m \mathcal{P}_mu = |u|^{p-1}u.
\end{equation}
Then, we look for a solution $u_m$ of \eqref{eq-main-tran} of the form
\begin{equation}\label{eq-2-0}
u_m=u_\infty+w.
\end{equation}
After inserting \eqref{eq-2-0} into the equation \eqref{eq-main-tran} and doing some algebraic manipulation, we get
\begin{equation}\label{eq-6-11}
\begin{split}
\left[ 2m\mathcal{P}_m - p|u_{\infty}|^{p-1}\right] w =& \Big\{ |u_{\infty} +w|^{p-1}(u_\infty+w) - |u_{\infty}|^{p-1}u_\infty - p|u_{\infty}|^{p-1}\Big\} 
\\
&\quad+ \left[ (-\Delta) - 2m\mathcal{P}_m\right]u_{\infty}.
\end{split}
\end{equation}
We now define operators
\begin{equation}\label{eq-6-3}
\left\{
\begin{split}
D_m(\phi) &:=  \left[ (-\Delta) - 2m\mathcal{P}_m\right]\phi,
\\
A_m(\phi) &:= \left[ I -\left(2m\mathcal{P}_m\right)^{-1} pu_{\infty}^{p-1} \right]\phi,
\\
Q(\phi) &:= |u_{\infty} +\phi|^{p-1}(u_\infty+\phi) - |u_{\infty}|^{p-1}u_\infty - p|u_{\infty}|^{p-1}\phi
\end{split}\right.
\end{equation}
acting on suitable function spaces on $\Omega$. We then, in a formal sense, arrive at the equation $w = \mathcal{L}_m \left[ Q(w) +D_m(u_{\infty})\right]$ from \eqref{eq-6-11} where we denote $\mathcal{L}_m := A_m^{-1}(2m\mathcal{P}_m)^{-1}$. We define
\[
\Phi_m(\phi) := \mathcal{L}_m \left[ Q(\phi) +D_m(u_{\infty})\right].
\]
For Theorem \ref{thm-3}, it suffices to show $\Phi$ is contractive on a small ball in $L^q(\Omega)$ for every $q \geq np$.
Indeed, if this is done, then we conclude that $u=u_\infty+w$, where $w$ is the (unique) fixed point of $\Phi$, is a strong solution of \eqref{eq-main-tran}.

We provide the following useful lemmas and propositions in order. Recall that $\{\lambda_n,\, \phi_n\}_{n=1}^\infty$ denotes the complete $L^2$ orthonormal system of eigenvalues and eigenfunctions of
\[
\left\{\begin{aligned}
-\Delta \phi &= \lambda \phi \quad \text{ in } \Omega \\
\phi &= 0 \quad \text{ on } \partial\Omega.
\end{aligned}\right.
\]

\begin{lem}\label{lem-6-0}
The inverse map $\mathcal{P}_m^{-1}$ given by
\[
\mathcal{P}_m^{-1}\phi := \sum_{i=1}^\infty \frac{c_i}{\sqrt{\lambda_i+m^2}-m}\phi_i, \quad \phi = \sum_{i=1}^\infty c_i\phi_i
\]
is a bounded map on $L^q(\Omega)$ to $W^{1,q}_0(\Omega)$ for any $q \in [2,\, \infty)$. In addition, there exists a positive constant $C_q$ independent of sufficiently large $m$ such that
\[
\|(2m\mathcal{P}_m)^{-1}\phi\|_{W^{1,q}_0} \leq C_q\|\phi\|_{L^q}. 
\]
\end{lem}
\begin{proof}
This follows from Theorem \ref{thm-a-4} and the fact that $\|\cdot\|_{W^{1,q}_0}$ is equivalent to $\|\sqrt{-\Delta}\cdot\|_{L^q}$. 
\end{proof}

From Lemma \ref{lem-6-0} and the fact that $u_\infty \in L^\infty(\Omega)$, we see that $A_m$ is well-defined and bounded on $L^q(\Omega)$ for $q \geq 2$.
The next lemma shows it is also invertible. 
\begin{lem}\label{lem-6-1}
Fix an arbitrary  $q \in [2,\, \infty)$. Then for sufficiently large $m > 0$, the operator $A_m$ is invertible on $L^q (\Omega)$.
In addition, there exists a constant $C_q >0$ independent of large $m$ such that
\begin{equation}
\|A_m^{-1} \phi\|_{L^q} \leq C_q \|\phi\|_{L^q}.
\end{equation}
\end{lem}
\begin{proof}
We first define an operator $A:L^q(\Omega)\to L^q(\Omega)$ as follows:
\[
A(\phi) := (I - K)(\phi), \quad K(\phi) := (-\Delta)^{-1}(p|u_{\infty}|^{p-1}\phi)
\]
Note that, due to the $W^{2,q}$ elliptic estimate and compact Sobolev embedding, the operator $K$ is compact on $L^q(\Omega)$.
Since $u_\infty$ is non-degenerate, the kernel of $A$ is trivial. 
Then the Fredholm altanative implies $A$ is invertible and its inverse $A^{-1}$ is bounded, i.e., there exists a constant $C > 0$ depending only on $n, p, q, \Omega$ and $u_{\infty}$ such that
\[
\|A^{-1}(\phi)\|_{L^q} \leq C \|\phi\|_{L^q}.
\]
We then compute that
\begin{equation}
\begin{split}
A_m & = \left( I - (-\Delta)^{-1} p|u_{\infty}|^{p-1} + \left( (-\Delta)^{-1} -(2m\mathcal{P}_m)^{-1}\right) p|u_{\infty}|^{p-1}\right)
\\
& = A \bigg[ I + A^{-1} ((-\Delta)^{-1} -(2m\mathcal{P}_m)^{-1}) p|u_{\infty}|^{p-1}\bigg]
\end{split}
\end{equation}
We claim that the operator norm 
$\left\|((-\Delta)^{-1} -(2m\mathcal{P}_m)^{-1}) p|u_{\infty}|^{p-1}\right\|_{\mathcal{L}(L^q)}$ can be made arbitrarily small as $m \to \infty$
so that $A_m$ is invertible on $L^q(\Omega)$ and $\|A_m^{-1}\|_{\mathcal{L}(L^q)}$ is independent of sufficiently large $m$.
Indeed, by \eqref{eq-a-6} in Theorem \ref{thm difference}, there exists a constant $C$ independent of large $m$ such that
\[
\|(-\Delta)^{-1}-(2m\mathcal{P}_m)^{-1}\|_{\mathcal{L}(L^q)}\leq\frac{C}{m^2},
\] 
which proves the claim. This completes the proof.
\end{proof}

Combining Lemma \ref{lem-6-0} and \ref{lem-6-1}, we get the following proposition.
\begin{prop}\label{prop-linear-est} For any $q \geq 2$, there exists a constant $C_q >0$ such that
\begin{equation}
\|\mathcal{L}_m \phi \|_{L^q (\Omega)} \leq C_q \|\phi\|_{L^n (\Omega)}\quad \text{for every } \phi \in L^n (\Omega).
\end{equation}
\end{prop}
\begin{proof}
This immediately follows from Lemma \ref{lem-6-0}, Lemma \ref{lem-6-1} and the embedding $W^{1,n}_0(\Omega) \hookrightarrow L^q(\Omega)$ for every $q \geq 2$.
\end{proof}


Next, we establish the estimates for the nonlinear operator $Q(w)$.
\begin{prop}\label{prop-nonlinear-est}
Let $q \geq np$. Then there exists a constant $C > 0$ depending only on $n,\, p,\, q$ and $\Omega$ such that the nonlinear operator $Q$ defined in \eqref{eq-6-3} satisfies the following estimates:
\begin{enumerate}
\item If $1<p\leq 2$, then
\begin{equation}
\|Q(w) - Q(v)\|_{L^n} \leq C \left( \|w\|_{L^q} + \|v\|_{L^q}\right)^{p-1}\|w-v\|_{L^q}.
\end{equation}
\item If $2 <p$, then
\begin{equation}
\|Q(w) - Q(v)\|_{L^n} \leq C \left( \|u_{\infty}\|_{L^q} + \|w\|_{L^q} + \|v\|_{L^q} \right)^{p-2} \left( \|w\|_{L^q} + \|v\|_{L^q}\right) \|w-v\|_{L^q}.
\end{equation}
\end{enumerate}
\end{prop}

\begin{proof}
By the fundamental theorem of calculus, we write
\begin{align*}
Q(w)-Q(v)&=\Big\{|u_\infty+w|^{p-1}(u_\infty+w)-|u_\infty+v|^{p-1}(u_\infty+v)\Big\}\\
&\quad\quad-p |u_\infty|^{p-1}(w-v)\\
&=\int_0^1 \frac{d}{dt}\Big[|u_\infty+(1-t)v+tw|^{p-1}(u_\infty+(1-t)v+tw)\Big]dt\\
&\quad\quad-p |u_\infty|^{p-1}(w-v)\\
&=p\int_0^1\big(|u_\infty+(1-t)v+tw|^{p-1}-|u_\infty|^{p-1}\big)(w-v)dt.
\end{align*}
Suppose that $1<p\leq2$. Then, by the elementary inequality
$$||a|^\ell-|b|^\ell|\leq \big||a|-|b|\big|^\ell\leq |a-b|^\ell\quad\textup{ if }0<\ell<1,$$
we have
$$|Q(w)-Q(v)|\leq C(|w|+|v|)^{p-1}|w-v|,$$
from which and the fact that $q \geq np$, we get the estimate 
\begin{equation}\label{nonlinear estimate proof 1}
\begin{aligned}
\|Q(w)-Q(v)\|_{L^n} \leq  C\|Q(w)-Q(v)\|_{L^{\frac{q}{p}}}
&\leq C\big(\|w\|_{L^{q}}+\|v\|_{L^{q}}\big)^{p-1}\|w-v\|_{L^{q}}.
\end{aligned}
\end{equation}
If $p>2$, using the fundamental theorem of calculus again, we find
$$|Q(w)-Q(v)|\leq C(|u_\infty|+|w|+|v|)^{p-2}(|w|+|v|)|w-v|.$$
From this and by estimating as above, we see that
\begin{equation}\label{nonlinear estimate proof 2}
\begin{aligned}
\|Q(w)-Q(v)\|_{L^{n}}&\leq C\|Q(w)-Q(v)\|_{L^{\frac{q}{p}}} \\
&\leq C\big(\|u_\infty\|_{L^{q}}+\|w\|_{L^{q}}+\|v\|_{L^{q}}\big)^{p-2}\big(\|w\|_{L^q}+\|v\|_{L^{q}}\big)\|w-v\|_{L^{q}}.
\end{aligned}
\end{equation}
Thus, the proposition is proved.
\end{proof}

By inserting $v \equiv 0$ into the above estimates, we obtain the following. 
\begin{cor}\label{cor-nonlinear-est}
Let $q \geq np$. Then operator $Q$ is a map from $L^q(\Omega)$ to $L^n(\Omega)$.
In addition, there exists a positive constant $C$ depending only on $n,\, p,\, q$ and $\Omega$ such that
\[
\left\{\begin{aligned}
\|Q(w)\|_{L^n} &\leq C\|w\|_{L^q}^p \quad \text{if } 1 < p \leq 2; \\
\|Q(w)\|_{L^n} &\leq C(\|u_\infty\|_{L^q}+\|w\|_{L^q})^{p-2}\|w\|_{L^q}^{2} \quad \text{if } 2 < p.
\end{aligned}\right.
\]
\end{cor}

Now we estimate the remainder term $D_m(u_\infty)$. 
\begin{prop}\label{prop-remainder} 
Let $q \geq 2$. Then there exists a constant $C > 0$ depending only on $n, q, p$ and $\Omega$ such that the following estimate holds:
\begin{equation}
\|D_m(u_\infty)\|_{L^q(\Omega)} \leq \left\{\begin{aligned} &\frac{C}{m^2}\|u_\infty\|_{W^{4,q}(\Omega)} \quad \text{if } p > 2; \\ &\frac{C}{m}\|u_\infty\|_{W^{3,q}(\Omega)} \quad \text{if } 1< p \leq 2. \end{aligned}\right.
\end{equation}
\end{prop}
\begin{proof}
Let $p > 2$. We invoke \eqref{eq-a-6} and \eqref{eq-a-8} in Theorem \ref{thm difference} to see that
\[
\begin{aligned}
\|D_m(u_\infty)\|_{L^q} & =  \left\|\left[ (-\Delta) - 2m\mathcal{P}_m\right]u_\infty\right\|_{L^q} = \left\|\left(\frac{1}{(-\Delta)}-\frac{1}{2m\mathcal{P}_m}\right)(-\Delta)(2m\mathcal{P}_m)u_\infty\right\|_{L^q} \\
& \leq \frac{C}{m^2}\|(-\Delta)(2m\mathcal{P}_m)u_\infty\|_{L^q} \leq \frac{C}{m^2}\|(-\Delta)^2u_\infty\|_{L^q} \leq \frac{C}{m^2}\|u_\infty\|_{W^{2,q}}.
\end{aligned}
\]
Similarly, if $1 < p \leq 2$, then by invoking \eqref{eq-a-7} and \eqref{eq-a-8} to get
\[
\begin{aligned}
\|D_m(u_\infty)\|_{L^q} \leq \frac{C}{m}\|\sqrt{-\Delta}(2m\mathcal{P}_m)u_\infty\|_{L^q} 
\leq \frac{C}{m}\|(\sqrt{-\Delta})^3u_\infty\|_{L^q} \leq \frac{C}{m}\|u_\infty\|_{W^{3,q}}.
\end{aligned}
\]


\end{proof}

By combining Proposition \ref{prop-linear-est}, Proposition \ref{prop-remainder} and Corollary \ref{cor-nonlinear-est}, we finally deduce that 
$\Phi_m$, given by $\mathcal{L}_m \left[ Q(\cdot) +D_m(u_{\infty})\right]$, is a contraction mapping on a small ball in $L^q(\Omega)$ whenever $q \geq np$.
Now we shall find a fixed point of $\Phi_m$ for sufficiently large $m > 0$. 
\begin{prop}\label{prop-fixed-point}
Let $q \geq np$. Then there exists a constant $m_0 > 1$ such that for $m \geq m_0$, the map $\Phi_m$ has a fixed point $w_m$, i.e.,
$w_m = \Phi_m(w_m)$.
\end{prop}
\begin{proof}
Let $\delta \in (0,1)$ be a small value to be chosen later. We consider the ball 
\begin{equation}
B_{\delta} = \left\{ u \in L^q (\Omega) ~:~ \|u\|_{L^q} \leq \delta \right\}.
\end{equation}
We aim to show that the map $\Phi_m$ is contractive on $B_{\delta}$ for large $m >1$. 
First, by combining Proposition \ref{prop-nonlinear-est} and Proposition \ref{prop-remainder}, we obtain
\begin{equation}
\|\Phi_m (w) - \Phi_m (v) \|_{L^q} \leq C \delta^{\min \{(p-1),1\}} \|w-v\|_{L^q}
\end{equation}
for some large constant $C$ independent of $\delta \in (0,1)$ and $m$. Similarly, we have
\begin{equation}\label{self-map}
\|\Phi_m (w)\|_{L^q} \leq C (\delta^{\min \{(p-1),1\}} \|w\|_{L^q} +\frac{1}{m^\alpha}),
\end{equation}
where $\alpha = 2$ if $p > 2$ and $\alpha = 1$ if $1 < p \leq 2$.
We first choose $\delta >0$ small so that $C \delta^{\min \{ (p-1),1\}} < \frac{1}{2}$, and then find $m_0 >1$ such that $\frac{C}{m_0^{\alpha}} < \frac{\delta}{2}$. 
Then, the mapping $\Phi$ is contractive from $B_{\delta}$ to itself. 
Therefore, by the contraction mapping principle, there exists a fixed point $w_m \in B_{\delta}$ of $\Phi_m$ when $m \geq m_0$.
This completes the proof.
\end{proof}

To complete the proof of Theorem \ref{thm-3}, it remains to show the following.
\begin{prop}
Let $q \geq np$ and a function $w_m \in L^q(\Omega)$ be a fixed point of $\Phi_m$, constructed in Proposition \ref{prop-fixed-point}. 
Then the following holds.
\begin{enumerate}
\item $w_m$ is contained in $W_0^{1,n}(\Omega)$ and there exists some constant $C > 0$ independent of $m$ such that
\[
\|w_m\|_{W^{1,n}_0} \leq \frac{C}{m^{\alpha}},
\] 
where $\alpha = 2$ if $p > 2$ and $\alpha = 1$ if $1< p\leq 2$; 
\item the function $u_m := w_m +u_{\infty}$ is a strong solution of \eqref{eq-main-tran}.
\end{enumerate}
\end{prop}
\begin{proof}
Since $w_m$ is a fixed point of $\Phi_m$, we see from Lemma \ref{lem-6-0}, Corollary \ref{cor-nonlinear-est} and Proposition \ref{prop-remainder} that
\[
\begin{aligned}
\|w_m\|_{W^{1,n}_0} & = \left\|(2m\mathcal{P}_m)^{-1}\left(p|u_\infty|^{p-1}w_m +Q(w_m)+D_m(u_\infty)\right)\right\|_{W^{1,n}_0} \\
&\leq C\left\|p|u_\infty|^{p-1}w_m +Q(w_m)+D_m(u_\infty)\right\|_{L^n} \\
&\leq C\left(\left\|w_m\right\|_{L^n} +\left\|Q(w_m)\right\|_{L^n}+ \left\|D_m(u_\infty)\right\|_{L^n}\right) \\
&\leq C\left(\left\|w_m\right\|_{L^q} +\left\|w_m\right\|_{L^q}^{\min\{2,p\}}+ \frac{1}{m^\alpha}\right).
\end{aligned}
\]
Also, the estimate \eqref{self-map} implies
\[
\|w_m\|_{L^q} = \|\Phi_m(w_m)\|_{L^q} \leq \frac12\|w\|_{L^q} + \frac{C}{m^\alpha}.
\]
Then, combining these two estimates, we get a proof of the first assertion $(i)$. 
To prove the assertion $(ii)$, we only need to check the function $u_m$ defined by $w_m +u_\infty$ satisfies the regularity assumptions; $u_m \in H^1_0(\Omega)$ and $|u_m|^{p-1}u_m \in L^2(\Omega)$. 
This easily can be seen from the embeddings $L^n(\Omega) \hookrightarrow L^2(\Omega)$ and $W^{1,n}_0(\Omega) \hookrightarrow L^{2p}(\Omega)$.  

\end{proof}

\appendix

\section{Norm estimates}
In this appendix, we give a proof of the boundedness of the operators in Section \ref{sec-ex-sup} employing the theorem of Duong-Sikora-Yan \cite{DSY}, which generalizes the classical H\"ormander-Mikhlin theorem. 
For a bounded Borel function $F:\R_+\to\R$, we define the operator $F(-\Delta)$ on $\Omega$ by
\[
F(-\Delta)\phi = \sum_{i=1}^\infty F(\lambda_i)c_i\phi_i, \quad \phi = \sum_{i=1}^\infty c_i\phi_i,
\] 
where $\{\lambda_n,\, \phi_n\}_{n=1}^\infty$ is the complete $L^2$ orthonormal system of eigenvalues and eigenfunctions of
\[
\left\{\begin{aligned}
-\Delta \phi &= \lambda \phi \quad \text{ in } \Omega \\
\phi &= 0 \quad \text{ on } \partial\Omega.
\end{aligned}\right.
\]

\begin{prop}[$\textup{\cite[Proposition 6.6]{DSY}}$]\label{prop-a-1}
Let $s> n/2$, $r_0 = \max (1,n/s)$ and $\eta \in C^\infty_c(\R_+)$ be a function not identically zero. 
Then for any bounded Borel function $F$ such that $\sup_{t>0} \|\eta F(t\cdot)\|_{W_s^{\infty}} < \infty$, 
the operator $F(-\Delta)$ is bounded on $L^p (\Omega)$ for all $p$ satisfying $r_0 <p<\infty$. In addition,
\begin{equation}
\|F(-\Delta)\|_{L^p (\Omega) \rightarrow L^p (\Omega)} \leq C_s \left( \sup_{t>0} \|\eta F(t\cdot)\|_{W_s^{\infty}} + |F(0)|\right).
\end{equation}
\end{prop}

To apply this result to our problem, we choose a smooth bump function $\chi :[0,+\infty) \rightarrow [0,+\infty)$ such that $\chi (s) =1$ for $s \in [0, \lambda_1 /2)$ and $\chi (s) = 0$ for $x \in (\lambda_1 , \infty)$, where $\lambda_1>0$ is the smallest positive eigenvalue of $(-\Delta)$ on $\Omega$. Then, we have $\chi (-\Delta ) =0$. Hence, our problem \eqref{eq-main} is equivalent to
\begin{equation}
\left( 2m\mathcal{P}_m + \chi (-\Delta) \right) u = |u|^{p-1} u.
\end{equation}



We consider functions $P_m:\R_+\to\R_+$ and $P_\infty:\R_+\to\R_+$ defined by
\begin{equation}\label{def: P_c, P_infty}
P_m(\lambda):= 2m(\sqrt{\lambda+m^2}-m)+ \chi (\lambda), \quad P_\infty(\lambda):= \lambda+\chi (\lambda)
\end{equation}
so that we have
$P_m(-\Delta) = 2m\mathcal{P}_m$ and $P_\infty(-\Delta) = -\Delta$.
Then the following symbol estimates can be shown. 

\begin{prop}\label{P_m comparison}\mbox{~}
\begin{enumerate}

\item 
For any positive integer $k$, there is a constant $C_k > 0$ such that for all $m \in [2,\infty)$,
\begin{equation}\label{eq-7-1}
\left|\left(\frac{d}{d\lambda}\right)^k \left(\frac{P_m(\lambda)}{P_\infty(\lambda)}\right)\right|\leq\frac{C_k}{\lambda^k}.
\end{equation}

\item 
For any positive integer $k$, there is a constant $C_k > 0$ such that for all $m \in [2,\infty)$,
\begin{equation}\label{eq-7-2}
\left|\left(\frac{d}{d\lambda}\right)^k \left(\frac{1}{P_\infty(\lambda)}-\frac{1}{P_m(\lambda)}\right)\right|
\leq\frac{C_k}{\lambda^k}\min\left\{\frac{1}{m^2}, \frac{1}{m(\lambda+1)^{\frac{1}{2}}}\right\}.
\end{equation}
\end{enumerate}
\end{prop}

\begin{proof}
The proposition can be proved exactly as in \cite[Proposition 2.3]{CHS}, thus we omit its proof. Indeed, the only difference comes from introducing bump function $\chi$
for controlling the singularities of $P_m^{-1}(\lambda)$ and $P_\infty^{-1}(\lambda)$ when we differentiate.
\end{proof}

As a consequence, as inverse operators, $P_m(-\Delta)$ converges to $P_\infty(-\Delta)$ as $m\to \infty$ in $L^q(\Omega)$ for any $q \geq 2$.
\begin{thm}[Difference between the inverses of two operators]\label{thm difference}
\mbox{~}
\begin{enumerate}
\item For $1<q<\infty$, there exists a constant $C_{q,n} >0$ such that
\begin{equation}\label{eq-a-6}
\left\|\left(\frac{1}{P_\infty(-\Delta)}-\frac{1}{P_m(-\Delta)}\right)f\right\|_{L^q}\leq \frac{C_{q,n}}{m^2}\|f\|_{L^q}\quad \forall\, m \in [2,\infty).
\end{equation}
\item For $1<q<\infty$, there exists a constant $C_{q,n} >0$ such that
\begin{equation}\label{eq-a-7}
\left\|\left(\frac{1}{P_\infty(-\Delta)}-\frac{1}{P_m(-\Delta)}\right)f\right\|_{L^q}
\leq \frac{C_{q,n}}{m}\left\|\frac{1}{P_\infty(-\Delta)^{\frac{1}{2}}}f\right\|_{L^q}\quad\forall\, m \in [2,\infty).
\end{equation}
\item For $1< q < \infty$, there exists a constant $C_{q,n} > 0$ such that
\begin{equation}\label{eq-a-8}
\left\|\frac{P_m(-\Delta)}{P_\infty(-\Delta)}f\right\|_{L^q}
\leq C_{q,n}\left\|f\right\|_{L^q}\quad\forall\, m \in [2,\infty).
\end{equation}
\end{enumerate}
\end{thm}
\begin{proof}
The estimates \eqref{eq-a-6} and \eqref{eq-a-8} immediately follow from Proposition \ref{prop-a-1}, \eqref{eq-7-1} and \eqref{eq-7-2}. 
To obtain \eqref{eq-a-7}, we compute for positive integer $k$,
\begin{align*}
&\left|\left(\frac{d}{d\lambda}\right)^k\left(\left(\frac{1}{P_\infty(\lambda)}-\frac{1}{P_m(\lambda)}\right)P_\infty(\lambda)^{\frac{1}{2}}\right)\right|
\\
&\qquad\leq\sum_{k_1+k_2=k} \left|\left(\frac{d}{d\lambda}\right)^{k_1}\left(\frac{1}{P_\infty(\lambda)}-\frac{1}{P_m(\lambda)}\right)\right|\left|\left(\frac{d}{d\lambda}\right)^{k_2}\left(P_\infty(\lambda)^{\frac{1}{2}}\right)\right|\\
&\qquad\leq \sum_{k_1+k_2=k}\frac{C_{k_1}}{m\lambda^{k_1+1}}\frac{C_{k_2}}{\lambda^{k_2-1}}= \sum_{k_1+k_2=k}\frac{C_{k_1}C_{k_2}}{m\lambda^{k}}\\
&\qquad\leq\frac{C_k}{m\lambda^{k}}.
\end{align*}
Then the estimate \eqref{eq-a-7} again follows from Proposition \ref{prop-a-1}.

\end{proof}

\begin{thm}\label{thm-a-4} For each $q\in (1,\infty)$, there exists a constant $C_q >0$ such that the following estimate holds:
\begin{align*}
\left\|\frac{\sqrt{-\Delta}}{P_m(-\Delta)}f\right\|_{L^q}\leq C_{q}\left\|f\right\|_{L^q}
\end{align*}
for any $f \in L^q (\Omega)$.
\end{thm}
\begin{proof}
By the triangle inequality and \eqref{eq-a-7}, we prove that 
\begin{align*}
\left\|\frac{\sqrt{-\Delta}}{P_m(-\Delta)}f\right\|_{L^q}&\leq \left\|\frac{\sqrt{-\Delta}}{P_\infty(-\Delta)}f\right\|_{L^q}+\left\|\left(\frac{\sqrt{-\Delta}}{P_\infty(-\Delta)}-\frac{\sqrt{-\Delta}}{P_m(-\Delta)}\right)f\right\|_{L^q}\\
&\leq \left\|\frac{\sqrt{-\Delta}}{P_\infty(-\Delta)}f\right\|_{L^q}+\frac{C_q}{m}\left\|\frac{\sqrt{-\Delta}}{P_\infty(-\Delta)^{\frac{1}{2}}}f\right\|_{L^q}\\
&\leq C_{q}\left\|f\right\|_{L^{q}}.
\end{align*}
The proof is done.
\end{proof}

\noindent \textbf{Acknowledgments}

This research of the first author was supported by Basic Science Research Program through the National Research Foundation of Korea(NRF) funded by the Ministry of Education (NRF-2017R1C1B5076348). 
This research of the second author was supported by Basic Science Research Program through the National Research Foundation of Korea(NRF) funded by the Ministry of Education (NRF-2017R1C1B1008215). 
This research of the third author was supported by Basic Science Research Program through the National Research Foundation of Korea(NRF) funded by the Ministry of Education (NRF-2017R1D1A1A09000768).

\end{document}